\documentclass[12pt,a4paper]{article}
\usepackage{}

\usepackage{latexsym}
\usepackage{amsmath}
\usepackage{amssymb}
\usepackage{arydshln}

\usepackage{cite}
\usepackage{stmaryrd}
\usepackage{enumerate}

\usepackage{hyperref}

\usepackage{color}
\usepackage{lineno}
\usepackage{graphicx}
\usepackage{ae}
\usepackage{amsmath}
\usepackage{amssymb}
\usepackage{latexsym}
\usepackage{url}
\usepackage{epsfig}
\usepackage{mathrsfs}
\usepackage{amsfonts}
\usepackage{amsthm}

\usepackage{float}
\usepackage{subfig}
\usepackage{tikz}
\usetikzlibrary{positioning}

\newcommand{\Cay}{\mathrm{Cay}}
\newcommand{\spec}{\mathrm{spec}}
\newcommand{\ICG}{\mathrm{ICG}}

\newcommand{\Per}{\mathcal{P}}

\newcommand{\CP}{\mathcal{J}}
\newtheorem{theorem}{Theorem}[section]

\newtheorem{lemma}[theorem]{Lemma}

\newtheorem{cor}[theorem]{Corollary}

\theoremstyle{definition}

\newtheorem{conjecture}{Conjecture}

\numberwithin{equation}{section} 
\allowdisplaybreaks    

\def\qed{\hfill$\Box$\vspace{12pt}}

\long\def\delete#1{}

\usepackage{xcolor}
\usepackage[normalem]{ulem}

\begin{document}
\title {So's conjecture for integral circulant graphs of $4$ types}

\author{Hao Li$^{a,b,c}$,~Xiaogang Liu$^{a,b,c,}$\thanks{Supported by the National Natural Science Foundation of China (No. 12371358) and the Guangdong Basic and Applied Basic Research Foundation (No. 2023A1515010986).}~$^,$\thanks{ Corresponding author. Email addresses: haoli1111@mail.nwpu.edu.cn, xiaogliu@nwpu.edu.cn.}~
\\[2mm]
{\small $^a$School of Mathematics and Statistics,}\\[-0.8ex]
{\small Northwestern Polytechnical University, Xi'an, Shaanxi 710072, P.R.~China}\\
{\small $^b$Research \& Development Institute of Northwestern Polytechnical University in Shenzhen,}\\[-0.8ex]
{\small Shenzhen, Guandong 518063, P.R. China}\\
{\small $^c$Xi'an-Budapest Joint Research Center for Combinatorics,}\\[-0.8ex]
{\small Northwestern Polytechnical University, Xi'an, Shaanxi 710129, P.R. China}\\
}

\date{}

\openup 0.5\jot
\maketitle

\begin{abstract}
In [\emph{Discrete Mathematics} 306 (2005) 153-158], So proposed a conjecture saying that integral circulant graphs with different connection sets have different spectra. This conjecture is still open. We prove that this conjecture holds for integral circulant graphs whose orders have prime factorization of $4$ types.



\smallskip


\end{abstract}
\section{Introduction}
\label{introduction}

The \emph{spectrum} of a graph $\Gamma$, denoted by $\spec(\Gamma)$, is the multiset of the eigenvalues of the adjacency matrix of $\Gamma$. Graphs are called \emph{isospectral} or \emph{cospectral} if they have the same spectrum. An \emph{integral graph} is a graph whose spectrum contains integers only. Let $G$ be a group, and let $S$ be a symmetric subset (that is, $\forall s\in S$, $s^{-1}\in S$) of $G$ without the identity. The \emph{Cayley graph} $\Cay(G,S)$ is defined to be the graph with vertex set $G$ and edges drawn from $g\in G$ to $h\in G$ whenever $hg^{-1}\in S$. The set $S$ is called the \emph{connection set} of $\Cay(G,S)$. In particular, if $G$ is a cyclic group, then $\Cay(G,S)$ is called a \emph{circulant graph}. We denote the \emph{ring of integers modulo $n$} by $\mathbb{Z}_{n}$. In this work, the spectrum of $\Cay(\mathbb{Z}_{n},S)$ is expressed in the following way:
\begin{align*}
\spec(\Cay(\mathbb{Z}_{n},S))=\left(
\begin{array}{cccc}
\nu_{1} & \nu_{2} & \ldots & \nu_{J} \\
m_{1} & m_{2} & \ldots & m_{J}
\end{array}
\right),
\end{align*}
where for every $j\in\{1,2,\ldots,J\}$, $\nu_{j}$ denotes distinct eigenvalues and $m_{j}$ denotes the multiplicity of $\nu_{j}$. Isomorphic circulant graphs do not necessarily have a common connection set. For example, $\Cay(\mathbb{Z}_{7},\{1,6\})$ and $\Cay(\mathbb{Z}_{7},\{2,5\})$ are both circuits of length $7$, but they have different connection sets. Since isomorphic graphs are isospectral, one may derive that isospectral circulant graphs do not necessarily share a connection set. However, this does not seem to be the case for integral circulant graphs. After computing the spectra of circulant graphs on less than $100$ vertices with all possible connection sets, So in \cite{so0} proposed the following conjecture.
\begin{conjecture}\emph{(See\cite[Conjecture~7.3]{so0})}
Let $\Cay(\mathbb{Z}_{n},S_{1})$ and $\Cay(\mathbb{Z}_{n},S_{2})$ be two integral circulant graphs. If $S_{1}\neq S_{2}$, then $\spec(\Cay(\mathbb{Z}_{n},S_{1}))\neq\spec(\Cay(\mathbb{Z}_{n},S_{2}))$, hence $\Cay(\mathbb{Z}_{n},S_{1})$ and $\Cay(\mathbb{Z}_{n},S_{2})$ are not isomorphic.
\end{conjecture}
With the above notation, Klin and Kov\'{a}cs in \cite{isomorphicconnectionset} pointed out that if $S_{1}\neq S_{2}$, then $\Cay(\mathbb{Z}_{n},S_{1})$ and $\Cay(\mathbb{Z}_{n},S_{2})$ are indeed not isomorphic, which follows directly from a conjecture of Toida\cite{toida}. However, the first part of So's conjecture, that is, the implication $S_{1}\neq S_{2}\Rightarrow\spec(\Cay(\mathbb{Z}_{n},S_{1}))=\spec(\Cay(\mathbb{Z}_{n},S_{2}))$ is still open.

So's conjecture is among studies on isospectrality of graphs and graphs determined by their spectra, which have been extensively studied\cite{3,sap,which,circulant,construction,368,newmethod,cospectral,282}. For more results, we refer the readers to the survey\cite[Section 4]{survey}. For a better description of existing results on integral circulant graphs, we introduce 3 notations and a lemma. Let $n\geq1$ be an integer. Set
$$
[n]=\{1,2,\ldots,n\}.
$$
For convenience, we ignore the distinction between $\mathbb{Z}_{n}$ and $[n]$. Let $d$ be a divisor of $n$. Set
$$
G_{n}(d)=\{j\in[n]:\gcd(j,n)=d\},
$$
where $\gcd(j,n)$ denotes the greatest common divisor of $j$ and $n$.
Besides, for any subset $S$ of $[n]$ which is a union of $G_{n}(d)$'s for some divisors $d$ of $n$, we denote by $\mathcal{D}_{S}$, the set of divisors of $n$ such that
$$
S=\bigcup_{d\in\mathcal{D}_{S}}G_{n}(d).
$$
Note that $\mathcal{D}_{S}$ depends not only on $S$ but on $n$ as well.
\begin{lemma}
\emph{(See\cite[Theorem~7.1]{so0})}
\label{integrality}
A circulant graph $\Cay(\mathbb{Z}_{n},S)$ is integral if and only if $S$ is a union of $G_{n}(d)$'s for some divisors $d$ of $n$.
\end{lemma}
Let $\Cay(\mathbb{Z}_{n},S)$ be an integral circulant graph. By Lemma \ref{integrality}, we have $\Cay(\mathbb{Z}_{n},S)=\Cay(\mathbb{Z}_{n},\bigcup_{d\in\mathcal{D}_{S}}G_{n}(d))$, which is determined by $n$ and $\mathcal{D}_{S}$. Hence, we denote an integral circulant graph $\Cay(\mathbb{Z}_{n},S)$ by $\ICG(n,\mathcal{D}_{S})$ for convenience. There hasn't been much progress in research on So's conjecture so far. We collected previous results on So's conjecture for integral circulant graphs in the following theorem.
\begin{theorem}\label{knownresult}
Let $n\geq1$ be an integer. Let $\ICG(n,\mathcal{D}_{S_{1}})$ and $\ICG(n,\mathcal{D}_{S_{2}})$ be two integral circulant graphs. $\spec(\ICG(n,\mathcal{D}_{S_{1}}))=\spec(\ICG(n,\mathcal{D}_{S_{2}}))$ implies $S_{1}=S_{2}$ if one of the following conditions is satisfied.
\begin{itemize}
\item[\rm (a)]$n=p^{k}$ or $n=pq$ with primes $2\leq p<q$ and $k\geq1$.\cite{so0}
\item[\rm (b)]$n=pq^{k}$ or $n=p^{2}q$ with primes $2\leq p<q$ and $k\geq1$.\cite{sap}
\item[\rm (c)]$n$ is square-free and both $\mathcal{D}_{S_{1}}$ and $\mathcal{D}_{S_{2}}$ contain exactly 2 prime factors of $n$.\cite{squarefree}
\item[\rm (d)]$n=pqr$ with primes $p<q<r$.\cite{howmany}
\end{itemize}

\end{theorem}


In this work, we continue to study on So's conjecture and verify $4$ cases where isospectrality implies sharing a connection set for integral circulant graphs. Here is our main theorem.
\begin{theorem}\label{main}
Let $n\geq1$ be an integer. Let $\ICG(n,\mathcal{D}_{S_{1}})$ and $\ICG(n,\mathcal{D}_{S_{2}})$ be two integral circulant graphs. $\spec(\ICG(n,\mathcal{D}_{S_{1}}))=\spec(\ICG(n,\mathcal{D}_{S_{2}}))$ implies $S_{1}=S_{2}$ if one of the following conditions is satisfied.
\begin{itemize}
\item[\rm (a)]$n\geq1$ is an odd integer with prime factorisation $n=p_{1}^{J_{1}}p_{2}^{J_{2}}\cdots p_{s}^{J_{s}}$ where $s\geq2$ and $\forall r\in[s-1]$, $\prod_{i=1}^{r}p_{i}^{J_{i}}<p_{r+1}$.

\item[\rm (b)]$n\geq1$ is an even integer with prime factorisation $n=2p_{1}^{J_{1}}p_{2}^{J_{2}}\cdots p_{s}^{J_{s}}$ where $s\geq2$ and $\forall r\in[s-1]$, $\prod_{i=1}^{r}p_{i}^{J_{i}}<p_{r+1}$.
\item[\rm (c)]$n=p^{3}q$ with primes $2\leq p<q$.
\item[\rm (d)]$n=p^{2}q^{2}$ with primes $2\leq p<q$.
\end{itemize}
\end{theorem}

Our paper is organized as follows. In Section \ref{preliminaries}, we introduce notations, definitions and useful results which will play important roles throughout the work. In Section \ref{proofofmaintheorem}, we give a proof of Theorem \ref{main}, which consists of Theorems \ref{maina} and \ref{mainb} in Subsection \ref{sub1} and Theorems \ref{233}, \ref{23q}, \ref{p3q}, \ref{2232}, \ref{22q2} and \ref{p2q2} in Subsection \ref{sub2}. In Section \ref{conclusion}, we conclude our work.


\section{Preliminaries}\label{preliminaries}
In this section, we introduce notations, definitions and useful results which will play important roles throughout the work.
\begin{itemize}
\item Let $n\geq1$ be an integer. Let $d$ be a divisor of $n$. We have
$$
G_{n}(d)=d\cdot G_{n/d}(1),
$$
where $d\cdot G_{n/d}(1)=\{dj:j\in G_{n/d}(1)\}$. What follows is that
$$
|G_{n}(d)|=|d\cdot G_{n/d}(1)|=|G_{n/d}(1)|=\phi(n/d),
$$
where $\phi$ is the \emph{Euler totient function} given by
\begin{align*}
\phi(k)=\left\{
\begin{array}{ll}
k\prod_{i=1}^{s}(1-\frac{1}{p_{i}}), & \text{if $k=p_{1}^{J_{1}}p_{2}^{J_{2}}\cdots p_{s}^{J_{s}}\geq2$ with primes $p_{1}<p_{2}<\cdots<p_{s}$},\\[0.2cm]
1, & \text{if $k=1$},
\end{array}
\right.
\end{align*}
for any integer $k\geq1$. Note that $\forall k\geq3$, $2|\phi(k)$. It is known (See\cite[Page~244, Theorem~7.7]{numbertheory}) that
\begin{equation}\label{phiproperty}
n=\sum_{d\in\mathcal{D}_{[n]}}\phi(d)=\sum_{d\in\mathcal{D}_{[n]}}\phi(n/d),
\end{equation}
where
$$\mathcal{D}_{[n]}=\{d\in[n]:d|n\}$$
according to our previous definition of $\mathcal{D}_{S}$ in Section \ref{introduction}.

\item Let $n\geq1$ be an integer. Let $S$ be a subset of $[n]$ such that $S$ is a union of $G_{n}(d)$'s for some divisors $d$ of $n$. We have
    \begin{equation}
    \label{Scard}
    |S|=|\bigcup_{d\in\mathcal{D}_{S}}G_{n}(d)|=\sum_{d\in\mathcal{D}_{S}}|G_{n}(d)|=\sum_{d\in\mathcal{D}_{S}}\phi(n/d).
    \end{equation}

\item Let $n\geq1$ be an integer. Set
$$
\{0,[n]\}=\{0,1,2,\ldots,n\}.
$$

\item Let $Y$ be a subset of $X$. Set $\chi_{Y}:X\rightarrow\{0,1\}$, such that
\begin{align*}
\chi_{Y}(x)=\left\{
\begin{array}{ll}
1, & \mathrm{if}\;x\in Y,\\[0.2cm]
0, & \mathrm{if}\;x\notin Y.
\end{array}
\right.
\end{align*}

\item Let $n\geq1$ be an integer. Set
$$
\omega_{n}=e^{2\pi\iota/n},
$$
where $\iota=\sqrt{-1}$ is the imaginary unit.
\end{itemize}

\begin{lemma}
\label{eigenvaluelambda}\emph{(See\cite[Corollary~3.2]{babai})}
The eigenvalues of $\Cay(\mathbb{Z}_{n},S)$ are given by $\lambda_{k}(S)$ for each $k\in[n]$, where $\lambda_{k}(S)$ are defined as
$$
\lambda_{k}(S)=\sum_{g\in[n]}\chi_{S}(g)\omega_{n}^{kg}=\sum_{g\in S}\omega_{n}^{kg}.
$$
\end{lemma}
By Lemma \ref{eigenvaluelambda}, the following lemma is obvious.
\begin{lemma}\label{lambdan}
Let $\Cay(\mathbb{Z}_{n},S_{1})$ and $\Cay(\mathbb{Z}_{n},S_{2})$ be isospectral circulant graphs. Then
$$
\lambda_{n}(S_{1})=\lambda_{n}(S_{2}).
$$
\end{lemma}

\begin{itemize}
\item Let $\Cay(\mathbb{Z}_{n},S)$ be a circulant graph. Let $\alpha$ be an eigenvalue of $\Cay(\mathbb{Z}_{n},S)$. Set
    $$
    \mathcal{L}_{S}(\alpha)=\{k\in[n]:\lambda_{k}(S)=\alpha\}.
    $$
    Given
    \begin{align*}
    \spec(\Cay(\mathbb{Z}_{n},S))=\left(
    \begin{array}{cccc}
    \nu_{1} & \nu_{2} & \ldots & \nu_{J} \\
    m_{1} & m_{2} & \ldots & m_{J}
    \end{array}
    \right),
    \end{align*}
    by Lemma \ref{eigenvaluelambda}, we have $\forall j\in[J]$,
    \begin{equation}\label{mj1}
    m_{j}=|\{k\in[n]:\lambda_{k}(S)=\nu_{j}\}|=|\mathcal{L}_{S}(\nu_{j})|.
    \end{equation}
\end{itemize}
\begin{lemma}\label{minusR}
Let $\Cay(\mathbb{Z}_{n},S_{1})$ and $\Cay(\mathbb{Z}_{n},S_{2})$ be isospectral circulant graphs sharing the spectrum
\begin{align*}
\left(
\begin{array}{cccc}
\nu_{1} & \nu_{2} & \ldots & \nu_{J} \\
m_{1} & m_{2} & \ldots & m_{J}
\end{array}
\right).
\end{align*}
Let $R$ be a subset of $[n]$ such that $\forall k\in R$, $\lambda_{k}(S_{1})=\lambda_{k}(S_{2})$. Then $\forall j\in[J]$,
$$
|\mathcal{L}_{S_{1}}(\nu_{j})\setminus R|=|\mathcal{L}_{S_{2}}(\nu_{j})\setminus R|.
$$
\end{lemma}
\begin{proof}
$\forall j\in[J]$, $\forall k\in\mathcal{L}_{S_{1}}(\nu_{j})\cap R$, $\nu_{j}=\lambda_{k}(S_{1})=\lambda_{k}(S_{2})$ and so $k\in\mathcal{L}_{S_{2}}(\nu_{j})\cap R$. Thus, $\mathcal{L}_{S_{1}}(\nu_{j})\cap R\subseteq\mathcal{L}_{S_{2}}(\nu_{j})\cap R$. Similarly, $\mathcal{L}_{S_{1}}(\nu_{j})\cap R\supseteq\mathcal{L}_{S_{2}}(\nu_{j})\cap R$ and so
\begin{equation}\label{minusReq1}
\mathcal{L}_{S_{1}}(\nu_{j})\cap R=\mathcal{L}_{S_{2}}(\nu_{j})\cap R.
\end{equation}
Then
\begin{align*}
|\mathcal{L}_{S_{1}}(\nu_{j})\setminus R|&=|\mathcal{L}_{S_{1}}(\nu_{j})|-|\mathcal{L}_{S_{1}}(\nu_{j})\cap R|
\\&=m_{j}-|\mathcal{L}_{S_{1}}(\nu_{j})\cap R|\tag{by (\ref{mj1})}
\\&=|\mathcal{L}_{S_{2}}(\nu_{j})|-|\mathcal{L}_{S_{1}}(\nu_{j})\cap R|\tag{by (\ref{mj1})}
\\&=|\mathcal{L}_{S_{2}}(\nu_{j})|-|\mathcal{L}_{S_{2}}(\nu_{j})\cap R|\tag{by (\ref{minusReq1})}
\\&=|\mathcal{L}_{S_{2}}(\nu_{j})\setminus R|.
\end{align*}
This completes the proof.\qed
\end{proof}

\begin{itemize}
\item Let $x\geq1$ and $y\geq1$ be integers. The \emph{Ramanujan sum} $\mathcal{R}_{x}(y)$ is defined as
$$
\mathcal{R}_{x}(y)=\sum_{g\in G_{x}(1)}\omega_{x}^{yg}.
$$
Given an integral circulant graph $\ICG(n,\mathcal{D}_{S})$, it is known (See\cite[Theorem~5.1]{so0}) that $\forall k\in[n]$,
\begin{equation}\label{ramanujan}
\lambda_{k}(S)=\sum_{d\in\mathcal{D}_{S}}\mathcal{R}_{n/d}(k).
\end{equation}
The following formula is given by Ramanujan in \cite{ramanujan},
\begin{equation}\label{lambdaphimu}
\mathcal{R}_{x}(y)=\frac{\phi(x)}{\phi(\frac{x}{\gcd(y,x)})}\mu(\frac{x}{\gcd(y,x)}),
\end{equation}
where $\mu$ is the \emph{M\"{o}bius function} given by
\begin{align*}
\mu(k)=\left\{
\begin{array}{ll}
1, & \text{if $k$ is square-free and has an even number of prime factors},\\
-1, & \text{if $k$ is square-free and has an odd number of prime factors},\\ 
0, & \text{if $k$ has a squared prime factor},
\end{array}
\right.
\end{align*}
for any integer $k\geq1$.

\end{itemize}
\begin{lemma}\label{sameorbiteigen}
Let $\ICG(n,\mathcal{D}_{S})$ be an integral circulant graph. Set $k_{1},k_{2}\in[n]$. If $k_{1},k_{2}$ satisfy $\gcd(k_{1},n)=\gcd(k_{2},n)$, then $\lambda_{k_{1}}(S)=\lambda_{k_{2}}(S)$.
\end{lemma}
\begin{proof}
For every divisor $d$ of $n$,
we have
\begin{equation}\label{sameorbiteigeneq1}
\gcd(k_{1},n/d)=\gcd(k_{2},n/d).
\end{equation}
Then
\begin{align*}
\lambda_{k_{1}}(S)&=\sum_{d\in\mathcal{D}_{S}}\frac{\phi(n/d)}{\phi(\frac{n/d}{\gcd(k_{1},n/d)})}\mu(\frac{n/d}{\gcd(k_{1},n/d)})\tag{by (\ref{ramanujan}) and (\ref{lambdaphimu})}
\\&=\sum_{d\in\mathcal{D}_{S}}\frac{\phi(n/d)}{\phi(\frac{n/d}{\gcd(k_{2},n/d)})}\mu(\frac{n/d}{\gcd(k_{2},n/d)})\tag{by (\ref{sameorbiteigeneq1})}
\\&=\lambda_{k_{2}}(S).\tag{by (\ref{ramanujan}) and (\ref{lambdaphimu})}
\end{align*}
This completes the proof.\qed
\end{proof}

\begin{cor}\label{orbiteigen}
Let $\ICG(n,\mathcal{D}_{S})$ be an integral circulant graph. Let $\alpha$ be an eigenvalue of $\ICG(n,\mathcal{D}_{S})$. Then $\mathcal{L}_{S}(\alpha)$ is a union of $G_{n}(d)$'s for some divisors $d$ of $n$.

\end{cor}

\begin{proof}
We give our proof by contradiction. Assume that $\mathcal{L}_{S}(\alpha)$ is not a union of $G_{n}(d)$'s for some divisors $d$ of $n$. Then there exists $d_{0}\in\mathcal{D}_{[n]}$ such that $\mathcal{L}_{S}(\alpha)\cap G_{n}(d_{0})\neq\emptyset$ and $G_{n}(d_{0})\setminus\mathcal{L}_{S}(\alpha)\neq\emptyset$. Set $k_{1}\in\mathcal{L}_{S}(\alpha)\cap G_{n}(d_{0})$ and $k_{2}\in G_{n}(d_{0})\setminus\mathcal{L}_{S}(\alpha)$. Since $k_{1},k_{2}\in G_{n}(d_{0})$, we have $\gcd(k_{1},n)=d_{0}=\gcd(k_{2},n)$. Hence by Lemma \ref{sameorbiteigen},
\begin{equation}\label{orbiteigeneq1}
\lambda_{k_{1}}(S)=\lambda_{k_{2}}(S).
\end{equation}
Since $k_{1}\in\mathcal{L}_{S}(\alpha)$ and $k_{2}\notin\mathcal{L}_{S}(\alpha)$, we have
\begin{align*}
\lambda_{k_{1}}(S)=\alpha\neq\lambda_{k_{2}}(S),
\end{align*}
which contradicts (\ref{orbiteigeneq1}).\qed
\end{proof}

\begin{lemma}
\label{lambdan/2}
Let $n\geq1$ be an integer such that $2|n$. Let $\ICG(n,\mathcal{D}_{S_{1}})$ and $\ICG(n,\mathcal{D}_{S_{2}})$ be isospectral integral circulant graphs. Then
$$\lambda_{\frac{n}{2}}(S_{1})=\lambda_{\frac{n}{2}}(S_{2}).$$
\end{lemma}


\begin{proof}
Set
\begin{align*}
\spec(\ICG(n,\mathcal{D}_{S_{1}}))=\spec(\ICG(n,\mathcal{D}_{S_{2}}))=\left(
\begin{array}{cccc}
\nu_{1} & \nu_{2} & \ldots & \nu_{J} \\
m_{1} & m_{2} & \ldots & m_{J}
\end{array}
\right).
\end{align*}
Set $\nu_{j_{0}}=\lambda_{\frac{n}{2}}(S_{1})$. Then
\begin{align*}
|\mathcal{L}_{S_{1}}(\nu_{j_{0}})\setminus\{n\}|
&=\sum_{d\in\mathcal{D}_{\mathcal{L}_{S_{1}}(\nu_{j_{0}})}\setminus\{n\}}\phi(n/d)\tag{by (\ref{Scard})}
\\&=\phi(2)+\sum_{d\in\mathcal{D}_{\mathcal{L}_{S_{1}}(\nu_{j_{0}})}\setminus\{\frac{n}{2},n\}}\phi(n/d)
\\&\equiv1~(\mathrm{mod}\; 2)\tag{because $\forall k\geq3$, $2|\phi(k)$}
\end{align*}
and $\forall j\in[J]\setminus\{j_{0}\}$,
\begin{align*}
|\mathcal{L}_{S_{1}}(\nu_{j})\setminus\{n\}|
&=\sum_{d\in\mathcal{D}_{\mathcal{L}_{S_{1}}(\nu_{j})}\setminus\{n\}}\phi(n/d)\tag{by (\ref{Scard})}
\\&=\sum_{d\in\mathcal{D}_{\mathcal{L}_{S_{1}}(\nu_{j})}\setminus\{\frac{n}{2},n\}}\phi(n/d)\tag{because $\lambda_{\frac{n}{2}}(S_{1})=\nu_{j_{0}}\neq\nu_{j}$}
\\&\equiv0~(\mathrm{mod}\; 2).\tag{because $\forall k\geq3$, $2|\phi(k)$}
\end{align*}
Set $\nu_{j_{0}'}=\lambda_{\frac{n}{2}}(S_{2})$. Similarly,
$$|\mathcal{L}_{S_{2}}(\nu_{j_{0}'})\setminus\{n\}|\equiv1~(\mathrm{mod}\; 2).$$
By Lemma \ref{lambdan}, $\lambda_{n}(S_{1})=\lambda_{n}(S_{2})$. By Lemma \ref{minusR}, $\forall j\in[J]$,
$$|\mathcal{L}_{S_{1}}(\nu_{j})\setminus\{n\}|=|\mathcal{L}_{S_{2}}(\nu_{j})\setminus\{n\}|.$$
In particular,
$$
|\mathcal{L}_{S_{1}}(\nu_{j_{0}'})\setminus\{n\}|=|\mathcal{L}_{S_{2}}(\nu_{j_{0}'})\setminus\{n\}|\equiv1~(\mathrm{mod}\;2).
$$
Therefore, $j_{0}=j_{0}'$ and so $\lambda_{\frac{n}{2}}(S_{1})=\nu_{j_{0}}=\nu_{j_{0}'}=\lambda_{\frac{n}{2}}(S_{2})$.\qed

\end{proof}


For convenience, we denote the set of odd integers by $\mathcal{O}$ and the set of even integers by $\mathcal{E}$.

\begin{cor}\label{n/2cor}
Let $n\geq1$ be an integer such that $2|n$. Let $\ICG(n,\mathcal{D}_{S_{1}})$ and $\ICG(n,\mathcal{D}_{S_{2}})$ be isospectral integral circulant graphs. Then we have
\begin{itemize}
\item[\rm{(a)}]$\sum_{d\in\mathcal{D}_{[n]}\cap\mathcal{O}}\chi_{\mathcal{D}_{S_{1}}}(d)\phi(n/d)
=\sum_{d\in\mathcal{D}_{[n]}\cap\mathcal{O}}\chi_{\mathcal{D}_{S_{2}}}(d)\phi(n/d)$; and
\item[\rm{(b)}]$\sum_{d\in\mathcal{D}_{[n]}\cap\mathcal{E}}\chi_{\mathcal{D}_{S_{1}}}(d)\phi(n/d)
=\sum_{d\in\mathcal{D}_{[n]}\cap\mathcal{E}}\chi_{\mathcal{D}_{S_{2}}}(d)\phi(n/d).$
\end{itemize}
\end{cor}
\begin{proof}
By (\ref{ramanujan}) and (\ref{lambdaphimu}),
$$
\lambda_{n}(S_{1})=\sum_{d\in\mathcal{D}_{S_{1}}}\phi(n/d)=\sum_{d\in\mathcal{D}_{[n]}\cap\mathcal{O}}\chi_{\mathcal{D}_{S_{1}}}(d)\phi(n/d)
+\sum_{d\in\mathcal{D}_{[n]}\cap\mathcal{E}}\chi_{\mathcal{D}_{S_{1}}}(d)\phi(n/d)
$$
and
$$
\lambda_{n}(S_{2})=
\sum_{d\in\mathcal{D}_{S_{2}}}\phi(n/d)
=\sum_{d\in\mathcal{D}_{[n]}\cap\mathcal{O}}\chi_{\mathcal{D}_{S_{2}}}(d)\phi(n/d)
+\sum_{d\in\mathcal{D}_{[n]}\cap\mathcal{E}}\chi_{\mathcal{D}_{S_{2}}}(d)\phi(n/d).
$$
By Lemma \ref{lambdan}, $\lambda_{n}(S_{1})=\lambda_{n}(S_{2})$ and so
\begin{equation}\label{n/2coreq1}\begin{split}
&\sum_{d\in\mathcal{D}_{[n]}\cap\mathcal{O}}\chi_{\mathcal{D}_{S_{1}}}(d)\phi(n/d)
+\sum_{d\in\mathcal{D}_{[n]}\cap\mathcal{E}}\chi_{\mathcal{D}_{S_{1}}}(d)\phi(n/d)
\\=&\sum_{d\in\mathcal{D}_{[n]}\cap\mathcal{O}}\chi_{\mathcal{D}_{S_{2}}}(d)\phi(n/d)
+\sum_{d\in\mathcal{D}_{[n]}\cap\mathcal{E}}\chi_{\mathcal{D}_{S_{2}}}(d)\phi(n/d).
\end{split}\end{equation}
By (\ref{ramanujan}) and (\ref{lambdaphimu}),
$$
\lambda_{\frac{n}{2}}(S_{1})=-\sum_{d\in\mathcal{D}_{[n]}\cap\mathcal{O}}\chi_{\mathcal{D}_{S_{1}}}(d)\phi(n/d)
+\sum_{d\in\mathcal{D}_{[n]}\cap\mathcal{E}}\chi_{\mathcal{D}_{S_{1}}}(d)\phi(n/d)
$$
and
$$
\lambda_{\frac{n}{2}}(S_{2})
=-\sum_{d\in\mathcal{D}_{[n]}\cap\mathcal{O}}\chi_{\mathcal{D}_{S_{2}}}(d)\phi(n/d)
+\sum_{d\in\mathcal{D}_{[n]}\cap\mathcal{E}}\chi_{\mathcal{D}_{S_{2}}}(d)\phi(n/d).
$$
By Lemma \ref{lambdan/2}, $\lambda_{\frac{n}{2}}(S_{1})=\lambda_{\frac{n}{2}}(S_{2})$ and so
\begin{equation}\label{n/2coreq2}\begin{split}
&-\sum_{d\in\mathcal{D}_{[n]}\cap\mathcal{O}}\chi_{\mathcal{D}_{S_{1}}}(d)\phi(n/d)
+\sum_{d\in\mathcal{D}_{[n]}\cap\mathcal{E}}\chi_{\mathcal{D}_{S_{1}}}(d)\phi(n/d)
\\=&-\sum_{d\in\mathcal{D}_{[n]}\cap\mathcal{O}}\chi_{\mathcal{D}_{S_{2}}}(d)\phi(n/d)
+\sum_{d\in\mathcal{D}_{[n]}\cap\mathcal{E}}\chi_{\mathcal{D}_{S_{2}}}(d)\phi(n/d).
\end{split}\end{equation}
By (\ref{n/2coreq1}) and (\ref{n/2coreq2}), we have (a) and (b).
This completes the proof.\qed

\end{proof}
Here we borrow a concept called ``super sequence'' from \cite{howmany}. Let $c>0$ be a real number. A sequence $\{x_{j}\}_{j=0}^{J}$ of positive real numbers is called a \emph{$c$-super sequence} if $\forall t\in[J]$, $x_{t}>c\sum_{j=0}^{t-1}x_{j}$. In particular, a super sequence in \cite{howmany} is a $1$-super sequence.

\begin{lemma}
\label{supersequence1}
Let $\{x_{j}\}_{j=0}^{J}$ be a $c$-super sequence. Let $\{a_{j}\}_{j=0}^{J}$ and $\{b_{j}\}_{j=0}^{J}$ be two finite sequences of nonnegative integers such that
\begin{itemize}
\item[\rm{(1)}] $\forall j\in\{0,[J]\}$, $0\leq a_{j},b_{j}\leq c$; and
\item[\rm{(2)}] $\sum_{j=0}^{J}a_{j}x_{j}=\sum_{j=0}^{J}b_{j}x_{j}$,
\end{itemize}
then $\forall j\in\{0,[J]\}$, $a_{j}=b_{j}$.
\end{lemma}

\begin{proof}
We give our proof by contradiction. Assume that $\exists j_{0}\in\{0,[J]\}$, s.t. $a_{j_{0}}\neq b_{j_{0}}$. Then $R=\{j\in\{0,[J]\}:a_{j}\neq b_{j}\}\neq\emptyset$. Let $M$ be the largest in $R$. Without loss of generality, set
$a_{M}>b_{M}$. Since both $a_{M}$ and $b_{M}$ are integers,
we have
\begin{equation}\label{supersequence1eq1}
a_{M}\geq1+b_{M}.
\end{equation}
Then
\begin{align*}
\sum_{j=0}^{J}a_{j}x_{j}&=\sum_{j=0}^{M}a_{j}x_{j}+\sum_{j=M+1}^{J}a_{j}x_{j}
\\&=\sum_{j=0}^{M}a_{j}x_{j}+\sum_{j=M+1}^{J}b_{j}x_{j}\tag{$M$ being the largest in $R$}
\\&\geq a_{M}x_{M}+\sum_{j=M+1}^{J}b_{j}x_{j}
\\&\geq(1+b_{M})x_{M}+\sum_{j=M+1}^{J}b_{j}x_{j}\tag{by (\ref{supersequence1eq1})}
\\&=x_{M}+b_{M}x_{M}+\sum_{j=M+1}^{J}b_{j}x_{j}
\\&>c\sum_{j=0}^{M-1}x_{j}+b_{M}x_{M}+\sum_{j=M+1}^{J}b_{j}x_{j}\tag{$\{x_{j}\}_{j=0}^{J}$ being a $c$-super sequence}
\\&=\sum_{j=0}^{M-1}cx_{j}+b_{M}x_{M}+\sum_{j=M+1}^{J}b_{j}x_{j}
\\&\geq\sum_{j=0}^{M-1}b_{j}x_{j}+b_{M}x_{M}+\sum_{j=M+1}^{J}b_{j}x_{j}\tag{by condition (1)}
\\&=\sum_{j=0}^{J}b_{j}x_{j},
\end{align*}
which contradicts condition (2).\qed
\end{proof}

\section{Proof of Theorem \ref{main}}\label{proofofmaintheorem}
In this section, we give a proof of Theorem \ref{main}.


\subsection{(a) and (b) of Theorem \ref{main}}\label{sub1}
In this subsection, we give proofs of (a) and (b) of Theorem \ref{main}.

\subsubsection{Useful notations and lemmas}
The notations and lemmas introduced here are important in the proofs of (a) and (b) of Theorem \ref{main}.
\begin{itemize}
\item Denote the set of positive real numbers by $\mathbb{R}_{+}$. Let $s\geq2$ be an integer. For every $i\in[s]$, let $J_{i}\geq1$ be an integer. Through the work, when a mapping
    $$
    f:\{(i,j):i\in[s],j\in\{0,[J_{i}]\}\}\rightarrow\mathbb{R}_{+}
    $$
    is given, we tacitly set for every $r\in[s]$,
    $$
    \mathcal{J}^{(r)}=\{0,[J_{1}]\}\times\{0,[J_{2}]\}\times\cdots\times\{0,[J_{r}]\},
    $$
    where $\times$ denotes the \emph{Cartesian product}. For convenience, we denote the image of $(i,j)$ under $f$ by $f_{i,j}$. Moreover, we tacitly set for every $r\in[s]$,
    \begin{equation}
    \label{defper}
    \Per^{(r)}=\sum_{\tau\in\CP^{(r)}}\prod_{i=1}^{r}f_{i,\tau_{i}},
    \end{equation}
    where $\tau_{i}$ denotes the $i$-th entry of $\tau$. In addition, we define $\Per^{(0)}=1$.


\end{itemize}

\begin{lemma}
\label{supersequence2}
Let $s\geq2$ be an integer. For every $i\in[s]$, let $J_{i}\geq1$ be an integer. Let
$$f:\{(i,j):i\in[s],j\in\{0,[J_{i}]\}\}\rightarrow\mathbb{R}_{+}$$
be a mapping such that $\forall r\in[s]$, $\{f_{r,j}\}_{j=0}^{J_{r}}$ is a $\Per^{(r-1)}$-super sequence. Let $A_{1}$ and $A_{2}$ be two subsets of $\CP^{(s)}$ such that
\begin{equation}\label{supersequence2eq1}
\sum_{\tau\in\CP^{(s)}}\chi_{A_{1}}(\tau)\prod_{i=1}^{s}f_{i,\tau_{i}}
=\sum_{\tau\in\CP^{(s)}}\chi_{A_{2}}(\tau)\prod_{i=1}^{s}f_{i,\tau_{i}}.
\end{equation}
Then $A_{1}=A_{2}$.
\end{lemma}

\begin{proof}
Set a mapping $$\eta:\tau=(\tau_{1},\tau_{2},\ldots,\tau_{s})\in\CP^{(s)}\mapsto(\tau_{1},\tau_{2},\ldots,\tau_{s-1})\in\CP^{(s-1)}.$$
For every $j\in\{0,[J_{s}]\}$, set
$$
\CP^{(s)}(j)=\{\tau\in\CP^{(s)}:\tau_{s}=j\}
$$
and
$$
A_{1}(j)=A_{1}\cap\CP^{(s)}(j),\;A_{2}(j)=A_{2}\cap\CP^{(s)}(j).
$$
In addition, setting
$$
a_{j}=\sum_{\tau\in\CP^{(s)}}\chi_{A_{1}(j)}(\tau)\prod_{i=1}^{s-1}f_{i,\tau_{i}},
$$
we have
\begin{align*}
a_{j}&=\sum_{\tau\in\CP^{(s)}(j)}\chi_{A_{1}(j)}(\tau)\prod_{i=1}^{s-1}f_{i,\tau_{i}}
\\&=\sum_{\sigma\in\CP^{(s-1)}}\chi_{\eta(A_{1}(j))}(\sigma)\prod_{i=1}^{s-1}f_{i,\sigma_{i}}\tag{$\eta$ being a bijection from $\mathcal{J}^{(s)}(j)$ to $\mathcal{J}^{(s-1)}$}
\\&\leq\sum_{\sigma\in\CP^{(s-1)}}\prod_{i=1}^{s-1}f_{i,\sigma_{i}}=\Per^{(s-1)}.\tag{by (\ref{defper})}
\end{align*}
Similarly, for every $j\in\{0,[J_{s}]\}$, set $b_{j}=\sum_{\tau\in\CP^{(s)}}\chi_{A_{2}(j)}(\tau)\prod_{i=1}^{s-1}f_{i,\tau_{i}}\leq\Per^{(s-1)}$. These two inequalities about $a_{j}$ and $b_{j}$ correspond to the condition (1) of Lemma \ref{supersequence1}. Moreover,
\begin{align*}
\sum_{\tau\in\CP^{(s)}}\chi_{A_{1}}(\tau)\prod_{i=1}^{s}f_{i,\tau_{i}}
&=\sum_{\tau\in \CP^{(s)}}\sum_{j=0}^{J_{s}}\chi_{A_{1}(j)}(\tau)(\prod_{i=1}^{s-1}f_{i,\tau_{i}})f_{s,j}
\\&=\sum_{j=0}^{J_{s}}\sum_{\tau\in \CP^{(s)}}\chi_{A_{1}(j)}(\tau)(\prod_{i=1}^{s-1}f_{i,\tau_{i}})f_{s,j}
\\&=\sum_{j=0}^{J_{s}}a_{j}f_{s,j}.
\end{align*}
Similarly,
$$
\sum_{\tau\in\CP^{(s)}}\chi_{A_{2}}(\tau)\prod_{i=1}^{s}f_{i,\tau_{i}}=\sum_{j=0}^{J_{s}}b_{j}f_{s,j}.
$$
By (\ref{supersequence2eq1}),
$$
\sum_{j=0}^{J_{s}}a_{j}f_{s,j}=\sum_{j=0}^{J_{s}}b_{j}f_{s,j},
$$
which is the condition (2) of Lemma \ref{supersequence1}. Recall the given condition that $\{f_{s,j}\}_{j=0}^{J_{s}}$ is a $\mathcal{P}^{(s-1)}$-super sequence.
By Lemma \ref{supersequence1}, $\forall j\in\{0,[J_{s}]\}$, $
a_{j}=b_{j}$, which means that
$$\sum_{\tau\in\CP^{(s)}}\chi_{A_{1}(j)}(\tau)\prod_{i=1}^{s-1}f_{i,\tau_{i}}
=\sum_{\tau\in\CP^{(s)}}\chi_{A_{2}(j)}(\tau)\prod_{i=1}^{s-1}f_{i,\tau_{i}}.$$
Note that $\forall j\in\{0,[J_{s}]\}$, $\eta$ is a bijection from $\CP^{(s)}(j)$ to $\CP^{(s-1)}$. We have $\forall j\in\{0,[J_{s}]\}$,
\begin{equation}\label{supersequence2eq2}
\sum_{\sigma\in\CP^{(s-1)}}\chi_{\eta(A_{1}(j))}(\sigma)\prod_{i=1}^{s-1}f_{i,\sigma_{i}}
=\sum_{\sigma\in\CP^{(s-1)}}\chi_{\eta(A_{2}(j))}(\sigma)\prod_{i=1}^{s-1}f_{i,\sigma_{i}}.
\end{equation}

In the following, we give our proof by induction on $s$. When $s=2$, (\ref{supersequence2eq2}) means that $\forall j\in\{0,[J_{2}]\}$,
$$
\sum_{k=0}^{J_{1}}\chi_{\eta(A_{1}(j))}(k)f_{1,k}=\sum_{k=0}^{J_{1}}\chi_{\eta(A_{2}(j))}(k)f_{1,k},
$$
which corresponds to the condition (2) of Lemma \ref{supersequence1}. Note that $\forall k\in\{0,[J_{1}]\}$, $$0\leq\chi_{\eta(A_{1}(j))}(k),\chi_{\eta(A_{2}(j))}(k)\leq1=\mathcal{P}^{(0)},$$
which corresponds to the condition (1) of Lemma \ref{supersequence1}. Recall the given condition that $\{f_{1,j}\}_{j=0}^{J_{1}}$ is a $\mathcal{P}^{(0)}$-super sequence. By Lemma \ref{supersequence1}, we have $\forall j\in\{0,[J_{2}]\}$, $\forall k\in\{0,[J_{1}]\}$, $\chi_{\eta(A_{1}(j))}(k)=\chi_{\eta(A_{2}(j))}(k)$ and so $\forall j\in\{0,[J_{2}]\}$, $A_{1}(j)=A_{2}(j)$. Then
$$A_{1}=\bigcup_{j=0}^{J_{2}}A_{1}(j)=\bigcup_{j=0}^{J_{2}}A_{2}(j)=A_{2}.$$
This completes the proof of the basis where $s=2$. Now assume that the assertion is true for $(s-1)$. For each $j\in\{0,[J_{s}]\}$, by (\ref{supersequence2eq2}),
using the induction hypothesis, we have $\eta(A_{1}(j))=\eta(A_{2}(j))$ and so $A_{1}(j)=A_{2}(j)$.
Then
$$A_{1}=\bigcup_{j=0}^{J_{s}}A_{1}(j)=\bigcup_{j=0}^{J_{s}}A_{2}(j)=A_{2}.$$ 
This completes the proof.\qed

\end{proof}







\begin{lemma}
\label{perphi1}
Let $n\geq1$ be an odd integer with prime factorisation $n=p_{1}^{J_{1}}p_{2}^{J_{2}}\cdots p_{s}^{J_{s}}$ where $s\geq2$. Set a mapping
$$
f:(i,j)\in\{(i,j):i\in[s],j\in\{0,[J_{i}]\}\}\mapsto\phi(p_{i}^{j})\in\mathbb{R}_{+}.
$$
Then
$$
\Per^{(s)}=n.
$$

\end{lemma}



\begin{proof}
Set a bijection
$$\psi:p_{1}^{k_{1}}p_{2}^{k_{2}}\cdots p_{s}^{k_{s}}\in\mathcal{D}_{[n]}\mapsto (k_{1},k_{2},\ldots,k_{s})\in\CP^{(s)}.$$
We have $\forall d\in\mathcal{D}_{[n]}$,
\begin{equation}\label{perphi1eq1}
\phi(d)=\prod_{i=1}^{s}\phi(p_{i}^{\psi(d)_{i}}),
\end{equation}
where $\psi(d)_{i}$ denotes the $i$-th entry of $\psi(d)$. Then
\begin{align*}
\Per^{(s)}
&=\sum_{\tau\in \CP^{(s)}}\prod_{i=1}^{s}f_{i,\tau_{i}}\tag{by (\ref{defper})}
\\&=\sum_{\tau\in\CP^{(s)}}\prod_{i=1}^{s}\phi(p_{i}^{\tau_{i}})
\\&=\sum_{d\in\mathcal{D}_{[n]}}\prod_{i=1}^{s}\phi(p_{i}^{\psi(d)_{i}})\tag{$\psi$ being a bijection}
\\&=\sum_{d\in\mathcal{D}_{[n]}}\phi(d)\tag{by (\ref{perphi1eq1})}
\\&=n\tag{by (\ref{phiproperty})}.
\end{align*}
This completes the proof.\qed

\end{proof}

Let $p$ be a prime and let $t\geq1$ be an integer. We have
\begin{equation}\label{propertyofphi1}
(p-1)\sum_{j=0}^{t-1}\phi(p^{j})=(p-1)[1+\sum_{j=1}^{t-1}p^{j-1}(p-1)]=(p-1)(1+p^{t-1}-1)=\phi(p^{t}).
\end{equation}

\begin{lemma}
\label{perphi2}
Let $n\geq1$ be an odd integer with prime factorisation $n=p_{1}^{J_{1}}p_{2}^{J_{2}}\cdots p_{s}^{J_{s}}$ where $s\geq2$ and $\forall r\in[s-1]$,
\begin{equation}\label{perphi2eq1}
\prod_{i=1}^{r}p_{i}^{J_{i}}<p_{r+1}.
\end{equation}
Set a mapping
$$
f:(i,j)\in\{(i,j):i\in[s],j\in\{0,[J_{i}]\}\}\mapsto\phi(p_{i}^{j})\in\mathbb{R}_{+}.
$$
Then $\forall r\in[s]$, $\{f_{r,j}\}_{j=0}^{J_{r}}$ is a $\mathcal{P}^{(r-1)}$-super sequence.
\end{lemma}
\begin{proof}
We first give the proof of the case where $r=1$. For any $t\in[J_{1}]$, by (\ref{propertyofphi1}), we have
\begin{align*}
\mathcal{P}^{(0)}\sum_{j=0}^{t-1}f_{1,j}=\sum_{j=0}^{t-1}f_{1,j}=\sum_{j=0}^{t-1}\phi(p_{1}^{j})<(p_{1}-1)\sum_{j=0}^{t-1}\phi(p_{1}^{j})=\phi(p_{1}^{t})
=f_{1,t}.
\end{align*}
We now give the proof of the case where $2\leq r\leq s$. For any $2\leq r\leq s$ and any $t\in[J_{r}]$, we have
\begin{align*}
\Per^{(r-1)}\sum_{j=0}^{t-1}f_{r,j}
&=\prod_{i=1}^{r-1}p_{i}^{J_{i}}\sum_{j=0}^{t-1}\phi(p_{r}^{j})\tag{by Lemma \ref{perphi1}}
\\&<(p_{r}-1)\sum_{j=0}^{t-1}\phi(p_{r}^{j})\tag{by (\ref{perphi2eq1}) and the assumption that all primes $p_{i}$ are odd}
\\&=\phi(p_{r}^{t})\tag{by (\ref{propertyofphi1})}
\\&=f_{r,t}.
\end{align*}
This completes the proof.\qed

\end{proof}


\begin{lemma}
\label{perphiodd}
Let $n\geq1$ be an odd integer with prime factorisation $n=p_{1}^{J_{1}}p_{2}^{J_{2}}\cdots p_{s}^{J_{s}}$ where $s\geq2$ and $\forall r\in[s-1]$, $\prod_{i=1}^{r}p_{i}^{J_{i}}<p_{r+1}$. Let $\mathcal{D}_{S_{1}}$ and $\mathcal{D}_{S_{2}}$ be two subsets of $\mathcal{D}_{[n]}$ such that
\begin{equation}\label{perphioddeq1}
\sum_{d\in\mathcal{D}_{[n]}}\chi_{\mathcal{D}_{S_{1}}}(d)\phi(n/d)
=\sum_{d\in\mathcal{D}_{[n]}}\chi_{\mathcal{D}_{S_{2}}}(d)\phi(n/d).
\end{equation}
Then $\mathcal{D}_{S_{1}}=\mathcal{D}_{S_{2}}$.

\end{lemma}

\begin{proof}
Set a mapping
$$
f:(i,j)\in\{(i,j):i\in[s],j\in\{0,[J_{i}]\}\}\mapsto\phi(p_{i}^{j})\in\mathbb{R}_{+}.
$$
Set a bijection
$$\psi:p_{1}^{l_{1}}p_{2}^{l_{2}}\cdots p_{s}^{l_{s}}\in\mathcal{D}_{[n]}\mapsto (J_{1}-l_{1},J_{2}-l_{2},\ldots,J_{s}-l_{s})\in\CP^{(s)}.$$
We have $\forall d\in\mathcal{D}_{[n]}$,
\begin{equation}\label{perphioddeq2}
n/d=\prod_{i=1}^{s}p_{i}^{\psi(d)_{i}}.
\end{equation}
Then
\begin{align*}
\sum_{d\in\mathcal{D}_{[n]}}\chi_{\mathcal{D}_{S_{1}}}(d)\phi(n/d)
&=\sum_{d\in\mathcal{D}_{[n]}}\chi_{\mathcal{D}_{S_{1}}}(d)\phi(\prod_{i=1}^{s}p_{i}^{\psi(d)_{i}})\tag{by the above equation}
\\&=\sum_{d\in\mathcal{D}_{[n]}}\chi_{\psi(\mathcal{D}_{S_{1}})}(\psi(d))\phi(\prod_{i=1}^{s}p_{i}^{\psi(d)_{i}})\tag{$\psi$ being a bijection}
\\&=\sum_{d\in\mathcal{D}_{[n]}}\chi_{\psi(\mathcal{D}_{S_{1}})}(\psi(d))\prod_{i=1}^{s}\phi(p_{i}^{\psi(d)_{i}})\tag{$\phi$ being a multplicative arithmetic function}
\\&=\sum_{\tau\in\CP^{(s)}}\chi_{\psi(\mathcal{D}_{S_{1}})}(\tau)\prod_{i=1}^{s}\phi(p_{i}^{\tau_{i}}),\tag{$\psi$ being a bijection}
\end{align*}
where $\psi(\mathcal{D}_{S_{1}})$ denotes the image of $\mathcal{D}_{S_{1}}$ under $\psi$. Similarly,
$$
\sum_{d\in\mathcal{D}_{[n]}}\chi_{\mathcal{D}_{S_{2}}}(d)\phi(n/d)
=\sum_{\tau\in\CP^{(s)}}\chi_{\psi(\mathcal{D}_{S_{2}})}(\tau)\prod_{i=1}^{s}\phi(p_{i}^{\tau_{i}}).
$$
By (\ref{perphioddeq1}),
$$
\sum_{\tau\in\CP^{(s)}}\chi_{\psi(\mathcal{D}_{S_{1}})}(\tau)\prod_{i=1}^{s}\phi(p_{i}^{\tau_{i}})
=\sum_{\tau\in\CP^{(s)}}\chi_{\psi(\mathcal{D}_{S_{2}})}(\tau)\prod_{i=1}^{s}\phi(p_{i}^{\tau_{i}}).
$$
By Lemmas \ref{perphi2} and \ref{supersequence2}, $\psi(\mathcal{D}_{S_{1}})=\psi(\mathcal{D}_{S_{2}})$ and so $\mathcal{D}_{S_{1}}=\mathcal{D}_{S_{2}}$.\qed


\end{proof}

\begin{lemma}
\label{perphieven}
Let $n\geq1$ be an even integer with prime factorisation $n=2p_{1}^{J_{1}}p_{2}^{J_{2}}\cdots p_{s}^{J_{s}}$ where $s\geq2$ and $\forall r\in[s-1]$, $\prod_{i=1}^{r}p_{i}^{J_{i}}<p_{r+1}$. Let $\mathcal{D}_{S_{1}}$ and $\mathcal{D}_{S_{2}}$ be two subsets of $\mathcal{D}_{[n]}$ such that
\begin{itemize}
\item[\rm{(1)}]
$
\sum\limits_{d\in\mathcal{D}_{[n]}\cap\mathcal{O}}\chi_{\mathcal{D}_{S_{1}}}(d)\phi(n/d)
=\sum\limits_{d\in\mathcal{D}_{[n]}\cap\mathcal{O}}\chi_{\mathcal{D}_{S_{2}}}(d)\phi(n/d)
$; and
\item[\rm{(2)}]
$
\sum\limits_{d\in\mathcal{D}_{[n]}\cap\mathcal{E}}\chi_{D_{S_{1}}}(d)\phi(n/d)
=\sum\limits_{d\in\mathcal{D}_{[n]}\cap\mathcal{E}}\chi_{D_{S_{2}}}(d)\phi(n/d)
$.
\end{itemize}
Then $\mathcal{D}_{S_{1}}=\mathcal{D}_{S_{2}}$.
\end{lemma}

\begin{proof}
We first prove that $\mathcal{D}_{S_{1}}\cap\mathcal{O}=\mathcal{D}_{S_{2}}\cap\mathcal{O}$. We have
\begin{align*}
\sum_{d\in\mathcal{D}_{[n]}\cap\mathcal{O}}\chi_{\mathcal{D}_{S_{1}}}(d)\phi(n/d)
&=\sum_{d\in\mathcal{D}_{[n]}\cap\mathcal{O}}\chi_{\mathcal{D}_{S_{1}}}(d)\phi(2n/2d)
\\&=\sum_{d\in\mathcal{D}_{[n]}\cap\mathcal{O}}\chi_{\mathcal{D}_{S_{1}}}(d)\phi(2)\phi(n/2d)\tag{$\phi$ being a multiplicative arithmetic function}
\\&=\sum_{d\in\mathcal{D}_{[n]}\cap\mathcal{O}}\chi_{\mathcal{D}_{S_{1}}}(d)\phi(n/2d).
\end{align*}
Similarly, we have
$$
\sum_{d\in\mathcal{D}_{[n]}\cap\mathcal{O}}\chi_{\mathcal{D}_{S_{2}}}(d)\phi(n/d)
=\sum_{d\in\mathcal{D}_{[n]}\cap\mathcal{O}}\chi_{\mathcal{D}_{S_{2}}}(d)\phi(n/2d).
$$
By condition (1),
$$
\sum_{d\in\mathcal{D}_{[n]}\cap\mathcal{O}}\chi_{\mathcal{D}_{S_{1}}}(d)\phi(n/2d)
=\sum_{d\in\mathcal{D}_{[n]}\cap\mathcal{O}}\chi_{\mathcal{D}_{S_{2}}}(d)\phi(n/2d).
$$
Note that $\mathcal{D}_{[n]}\cap\mathcal{O}=\{d\in[n/2]:d|(n/2)\}$. By Lemma \ref{perphiodd}, $\mathcal{D}_{S_{1}}\cap\mathcal{O}=\mathcal{D}_{S_{2}}\cap\mathcal{O}$.

We now prove that $\mathcal{D}_{S_{1}}\cap\mathcal{E}=\mathcal{D}_{S_{2}}\cap\mathcal{E}$. We have
\begin{align*}
\sum_{d\in\mathcal{D}_{[n]}\cap\mathcal{E}}\chi_{\mathcal{D}_{S_{1}}}(d)\phi(n/d)
&=\sum_{d\in2\cdot(\mathcal{D}_{[n]}\cap\mathcal{O})}\chi_{\mathcal{D}_{S_{1}}}(d)\phi(n/d)
\\&=\sum_{d\in\mathcal{D}_{[n]}\cap\mathcal{O}}\chi_{\mathcal{D}_{S_{1}}}(2d)\phi(n/2d)
\\&=\sum_{d\in\mathcal{D}_{[n]}\cap\mathcal{O}}\chi_{\frac{1}{2}\cdot(\mathcal{D}_{S_{1}}\cap\mathcal{E})}(d)\phi(n/2d).
\end{align*}
Similarly, we have
$$
\sum_{d\in\mathcal{D}_{[n]}\cap\mathcal{E}}\chi_{\mathcal{D}_{S_{2}}}(d)\phi(n/d)
=\sum_{d\in\mathcal{D}_{[n]}\cap\mathcal{O}}\chi_{\frac{1}{2}\cdot(\mathcal{D}_{S_{2}}\cap\mathcal{E})}(d)\phi(n/2d).
$$
By condition (2),
$$
\sum_{d\in\mathcal{D}_{[n]}\cap\mathcal{O}}\chi_{\frac{1}{2}\cdot(\mathcal{D}_{S_{1}}\cap\mathcal{E})}(d)\phi(n/2d)
=\sum_{d\in\mathcal{D}_{[n]}\cap\mathcal{O}}\chi_{\frac{1}{2}\cdot(\mathcal{D}_{S_{2}}\cap\mathcal{E})}(d)\phi(n/2d).
$$
By Lemma \ref{perphiodd}, we have $\frac{1}{2}\cdot(\mathcal{D}_{S_{1}}\cap\mathcal{E})=\frac{1}{2}\cdot(\mathcal{D}_{S_{2}}\cap\mathcal{E})$ and so $\mathcal{D}_{S_{1}}\cap\mathcal{E}=\mathcal{D}_{S_{2}}\cap\mathcal{E}$.\qed

\end{proof}

\subsubsection{Proofs of (a) and (b) of Theorem \ref{main}}
Here we give proofs of (a) and (b) of Theorem \ref{main}.
\begin{theorem}\emph{(Theorem \ref{main} (a))}\label{maina}
Let $n\geq1$ be an odd integer with prime factorisation $n=p_{1}^{J_{1}}p_{2}^{J_{2}}\cdots p_{s}^{J_{s}}$ where $s\geq2$ and $\forall r\in[s-1]$, $\prod_{i=1}^{r}p_{i}^{J_{i}}<p_{r+1}$. Let $\mathcal{D}_{S_{1}}$ and $\mathcal{D}_{S_{2}}$ be two subsets of $\mathcal{D}_{[n]}\setminus\{n\}$. Then $\spec(\ICG(n,\mathcal{D}_{S_{1}}))=\spec(\ICG(n,\mathcal{D}_{S_{2}}))$ implies $\mathcal{D}_{S_{1}}=\mathcal{D}_{S_{2}}$.
\end{theorem}

\begin{proof}
By Lemma \ref{lambdan}, $\lambda_{n}(S_{1})=\lambda_{n}(S_{2})$. By (\ref{ramanujan}) and (\ref{lambdaphimu}),
$$
\sum_{d\in\mathcal{D}_{[n]}}\chi_{\mathcal{D}_{S_{1}}}(d)\phi(n/d)
=\sum_{d\in\mathcal{D}_{[n]}}\chi_{\mathcal{D}_{S_{2}}}(d)\phi(n/d).
$$
By Lemma \ref{perphiodd}, $\mathcal{D}_{S_{1}}=\mathcal{D}_{S_{2}}$.\qed
\end{proof}

\begin{theorem}\emph{(Theorem \ref{main} (b))}\label{mainb}
Let $n\geq1$ be an integer with prime factorisation $n=2p_{1}^{J_{1}}p_{2}^{J_{2}}\cdots p_{s}^{J_{s}}$ where $s\geq2$ and $\forall r\in[s-1]$, $\prod_{i=1}^{r}p_{i}^{J_{i}}<p_{r+1}$. Let $\mathcal{D}_{S_{1}}$ and $\mathcal{D}_{S_{2}}$ be two subsets of $\mathcal{D}_{[n]}\setminus\{n\}$. Then $\spec(\ICG(n,\mathcal{D}_{S_{1}}))=\spec(\ICG(n,\mathcal{D}_{S_{2}}))$ implies $\mathcal{D}_{S_{1}}=\mathcal{D}_{S_{2}}$.

\end{theorem}

\begin{proof}
By Corollary \ref{n/2cor}, we have
$$
\sum_{d\in\mathcal{D}_{[n]}\cap\mathcal{O}}\chi_{\mathcal{D}_{S_{1}}}(d)\phi(n/d)
=\sum_{d\in\mathcal{D}_{[n]}\cap\mathcal{O}}\chi_{\mathcal{D}_{S_{2}}}(d)\phi(n/d)
$$
and
$$
\sum_{d\in\mathcal{D}_{[n]}\cap\mathcal{E}}\chi_{\mathcal{D}_{S_{1}}}(d)\phi(n/d)
=\sum_{d\in\mathcal{D}_{[n]}\cap\mathcal{E}}\chi_{\mathcal{D}_{S_{2}}}(d)\phi(n/d).
$$
By Lemma \ref{perphieven}, $\mathcal{D}_{S_{1}}=\mathcal{D}_{S_{2}}$.\qed
\end{proof}

\subsection{(c) and (d) of Theorem \ref{main}}\label{sub2}
In this subsection, we give proofs of (c) and (d) of Theorem \ref{main}.

\subsubsection{Useful notations and lemmas}
The notations and lemmas introduced here are important in the proofs of (c) and (d) of Theorem \ref{main}.
\begin{lemma}
\label{n/2notintriangle}
Let $n\geq1$ be an integer. Let $\mathcal{D}_{S_{1}}$ and $\mathcal{D}_{S_{2}}$ be two subsets of $\mathcal{D}_{[n]}\setminus\{n\}$. If $\spec(\ICG(n,\mathcal{D}_{S_{1}}))=\spec(\ICG(n,\mathcal{D}_{S_{2}}))$, then $n/2\notin\mathcal{D}_{S_{1}}\triangle\mathcal{D}_{S_{2}}$, where $\triangle$ denotes the \emph{symmetric difference}.
\end{lemma}
\begin{proof}
We give our proof by contradiction. Assume that $n/2\in\mathcal{D}_{S_{1}}\triangle\mathcal{D}_{S_{2}}$, without loss of generality, we set $n/2\in\mathcal{D}_{S_{1}}\setminus\mathcal{D}_{S_{2}}$. Since $\forall d\in\mathcal{D}_{[n]}\setminus\{n,n/2\}$, $2|\phi(n/d)$, we have
\begin{align*}
\sum_{d\in\mathcal{D}_{S_{1}}\setminus\mathcal{D}_{S_{2}}}\phi(n/d)
&=\phi(2)+\sum_{d\in\mathcal{D}_{S_{1}}\setminus\mathcal{D}_{S_{2}}\setminus\{n/2\}}\phi(n/d)
\equiv1~(\mathrm{mod}\; 2)
\end{align*}
and
\begin{align*}
\sum_{d\in\mathcal{D}_{S_{2}}\setminus\mathcal{D}_{S_{1}}}\phi(n/d)
&\equiv0~(\mathrm{mod}\; 2).
\end{align*}
By Lemma \ref{lambdan}, $\lambda_{n}(S_{1})=\lambda_{n}(S_{2})$. By (\ref{ramanujan}) and (\ref{lambdaphimu}), $\sum_{d\in\mathcal{D}_{S_{1}}}\phi(n/d)=\sum_{d\in\mathcal{D}_{S_{2}}}\phi(n/d)$ and so $\sum_{d\in\mathcal{D}_{S_{1}}\setminus\mathcal{D}_{S_{2}}}\phi(n/d)
=\sum_{d\in\mathcal{D}_{S_{2}}\setminus\mathcal{D}_{S_{1}}}\phi(n/d)$, leading to
$$
\sum_{d\in\mathcal{D}_{S_{1}}\setminus\mathcal{D}_{S_{2}}}\phi(n/d)
\equiv\sum_{d\in\mathcal{D}_{S_{2}}\setminus\mathcal{D}_{S_{1}}}\phi(n/d)~(\mathrm{mod}\; 2),
$$
which is a contradiction.\qed

\end{proof}
\begin{lemma}
\label{n/pnotintriangle}\emph{(See\cite[Theorem~3.3.10]{n/p})}
Let $n\geq1$ be an integer. Let $\mathcal{D}_{S_{1}}$ and $\mathcal{D}_{S_{2}}$ be two subsets of $\mathcal{D}_{[n]}\setminus\{n\}$. If $\spec(\ICG(n,\mathcal{D}_{S_{1}}))=\spec(\ICG(n,\mathcal{D}_{S_{2}}))$, then for any odd prime divisor $p$ of $n$, $n/p\notin\mathcal{D}_{S_{1}}\triangle\mathcal{D}_{S_{2}}$.
\end{lemma}


\begin{lemma}\label{notincontradiction}
Let $A$ be a finite nonempty set. Let $f$ be a mapping $f:A\rightarrow\{0\}\cup\mathbb{R}_{+}$ with a nonempty subset $B\subseteq A$ such that
\begin{equation}\label{notincontradictioneq1}
\sum_{a\in B}f(a)>\sum_{a\in A\setminus B}f(a).
\end{equation}
Let $A_{1}$ and $A_{2}$ be two subsets of $A$ such that $A_{1}\cap A_{2}=\emptyset$ and
\begin{equation}\label{notincontradictioneq2}
\sum_{a\in A_{1}}f(a)=\sum_{a\in A_{2}}f(a).
\end{equation}
Then we have $B\nsubseteq A_{1}$ and $B\nsubseteq A_{2}$. In particular, if $B=\{b\}$ has only one element, then $b\notin A_{1}\cup A_{2}$.



\end{lemma}
\begin{proof}
We give our proof by contradiction. Without loss of generality, assuming that $B\subseteq A_{1}$, we have
\begin{equation}\label{notincontradictioneq3}
A_{2}\subseteq A\setminus B.
\end{equation}
Then
\begin{align*}
\sum_{a\in A_{2}}f(a)&=\sum_{a\in A_{1}}f(a)\tag{by (\ref{notincontradictioneq2})}
\\&=\sum_{a\in B}f(a)+\sum_{a\in A_{1}\setminus B}f(a)
\\&\geq\sum_{a\in B}f(a)
\\&>\sum_{a\in A\setminus B}f(a)\tag{by (\ref{notincontradictioneq1})}
\\&\geq\sum_{a\in A_{2}}f(a),\tag{by (\ref{notincontradictioneq3})}
\end{align*}
which is a contradiction.\qed
\end{proof}

\begin{lemma}\label{notincontradictionnew}
Let $A=\{a_{1}\}$ be a set having only one element. Let $f:A\rightarrow\mathbb{R}_{+}$ be a mapping. Let $A_{1}$ and $A_{2}$ be two subsets of $A$ such that $A_{1}\cap A_{2}=\emptyset$ and
\begin{equation}\label{notincontradictionneweq1}
\sum_{a\in A_{1}}f(a)=\sum_{a\in A_{2}}f(a).
\end{equation}
Then $A_{1}\cup A_{2}=\emptyset$.
\end{lemma}
\begin{proof}
We give our proof by contradiction. Assume that $A_{1}\cup A_{2}\neq\emptyset$. Then $A_{1}\cup A_{2}=\{a_{1}\}$. Without loss of generality, suppose that $A_{1}=\{a_{1}\}$. Since $A_{1}\cap A_{2}=\emptyset$, $A_{2}=\emptyset$. Then $\sum_{a\in A_{1}}f(a)=f(a_{1})>0=\sum_{a\in A_{2}}f(a)$, contradicting (\ref{notincontradictionneweq1}).\qed

\end{proof}

\begin{cor}\label{notincontradictioncor}
Let $A=\{a_{1},a_{2}\}$ be a set having only two elements. Let $f$ be a mapping $f:A\rightarrow\mathbb{R}_{+}$. Let $A_{1}$ and $A_{2}$ be two subsets of $A$ such that $A_{1}\cap A_{2}=\emptyset$ and
$$
\sum_{a\in A_{1}}f(a)=\sum_{a\in A_{2}}f(a).
$$
Then either
\begin{itemize}
\item[\rm{(a)}]$f(a_{1})=f(a_{2})$ and $|A_{1}|=|A_{2}|=1$; or
\item[\rm{(b)}]$A_{1}\cup A_{2}=\emptyset$.
\end{itemize}
\end{cor}
\begin{proof}
We first rule out the case where $|A_{1}\cup A_{2}|=1$ by contradiction. Assume that $|A_{1}\cup A_{2}|=1$. Taking $A_{1}\cup A_{2}$, $f$, $A_{1}$, $A_{2}$, as $A$, $f$, $A_{1}$, $A_{2}$, in Lemma \ref{notincontradictionnew}, we have $A_{1}\cup A_{2}=\emptyset$, contradicting $|A_{1}\cup A_{2}|=1$.

Now, we have either $A_{1}\cup A_{2}=\{a_{1},a_{2}\}$ or $A_{1}\cup A_{2}=\emptyset$. It remains to prove that when $A_{1}\cup A_{2}=\{a_{1},a_{2}\}$, we have $f(a_{1})=f(a_{2})$ and $|A_{1}|=|A_{2}|=1$.

Suppose $A_{1}\cup A_{2}=\{a_{1},a_{2}\}$. We first prove that $f(a_{1})=f(a_{2})$. We give our proof by contradiction. Assume that $f(a_{1})\neq f(a_{2})$. Without loss of generality, suppose that $f(a_{1})>f(a_{2})$. Taking $A$, $f$, $\{a_{1}\}$, $A_{1}$, $A_{2}$, as $A$, $f$, $B$, $A_{1}$, $A_{2}$, in Lemma \ref{notincontradiction}, we have $a_{1}\notin A_{1}\cup A_{2}$, which is a contradiction.

We now prove that $|A_{1}|=|A_{2}|=1$. We give our proof by contradiction. Without loss of generality, assume that $|A_{1}|=2$. Then $A_{1}=\{a_{1},a_{2}\}$ and $A_{2}=\emptyset$. And so
$$
\sum_{a\in A_{1}}f(a)=f(a_{1})+f(a_{2})>0=\sum_{a\in A_{2}}f(a),
$$
contradicting (\ref{notincontradictioneq2}).\qed


\end{proof}



\begin{lemma}
\label{1notintriangle}
Let $n\geq1$ be an integer such that $\phi(n)\geq n/2$. Let $\mathcal{D}_{S_{1}}$ and $\mathcal{D}_{S_{2}}$ be two subsets of $\mathcal{D}_{[n]}\setminus\{n\}$. If $\spec(\ICG(n,\mathcal{D}_{S_{1}}))=\spec(\ICG(n,\mathcal{D}_{S_{2}}))$, 
then $1\notin\mathcal{D}_{S_{1}}\triangle\mathcal{D}_{S_{2}}$.
\end{lemma}
\begin{proof}
Since $\phi(n)\geq n/2$, we have
$$\phi(n/1)\geq n-\phi(n/1)>n-1-\phi(n/1)
=\sum_{d\in(\mathcal{D}_{[n]}\setminus\{n\})\setminus\{1\}}\phi(n/d).$$
By Lemma \ref{lambdan}, $\lambda_{n}(S_{1})=\lambda_{n}(S_{2})$. By (\ref{ramanujan}) and (\ref{lambdaphimu}), $\sum_{d\in\mathcal{D}_{S_{1}}}\phi(n/d)
=\sum_{d\in\mathcal{D}_{S_{2}}}\phi(n/d)$ and so
$$
\sum_{d\in\mathcal{D}_{S_{1}}\setminus\mathcal{D}_{S_{2}}}\phi(n/d)
=\sum_{d\in\mathcal{D}_{S_{2}}\setminus\mathcal{D}_{S_{1}}}\phi(n/d).$$
Taking $\mathcal{D}_{[n]}\setminus\{n\}$, $\phi(n/\cdot)$, $\{1\}$, $\mathcal{D}_{S_{1}}\setminus\mathcal{D}_{S_{2}}$, $\mathcal{D}_{S_{2}}\setminus\mathcal{D}_{S_{1}}$, as $A$, $f$, $B$, $A_{1}$, $A_{2}$, in Lemma \ref{notincontradiction}, we have $1\notin\mathcal{D}_{S_{1}}\triangle\mathcal{D}_{S_{2}}$.\qed

\end{proof}

Let $n\geq1$ be an integer. Let $\mathcal{D}_{S}$ be a subset of $\mathcal{D}_{[n]}$. $\overline{\mathcal{D}_{S}}$ is defined as
$$
\overline{\mathcal{D}_{S}}=\mathcal{D}_{[n]}\setminus\{n\}\setminus\mathcal{D}_{S}.
$$


\begin{lemma}
\label{regulariso}
Let $n\geq1$ be an integer. Let $\mathcal{D}_{S_{1}}$ and $\mathcal{D}_{S_{2}}$ be two subsets of $\mathcal{D}_{[n]}\setminus\{n\}$. If $\spec(\ICG(n,\mathcal{D}_{S_{1}}))=\spec(\ICG(n,\mathcal{D}_{S_{2}}))$, then $\spec(\ICG(n,\overline{\mathcal{D}_{S_{1}}}))=\spec(\ICG(n,\overline{\mathcal{D}_{S_{2}}}))$.

\end{lemma}
\begin{proof}
Since both $\ICG(n,\mathcal{D}_{S_{1}})$ and $\ICG(n,\mathcal{D}_{S_{2}})$ are regular, $\spec(\ICG(n,\overline{\mathcal{D}_{S_{1}}}))$ and $\spec(\ICG(n,\overline{\mathcal{D}_{S_{2}}}))$ are determined by $\ICG(n,\mathcal{D}_{S_{1}})$ and $\ICG(n,\mathcal{D}_{S_{2}})$, respectively.\qed
\end{proof}

\begin{cor}
\label{regularisocor}
Let $n\geq1$ be an integer. Set $d_{0}\in\mathcal{D}_{[n]}\setminus\{n\}$. Then $(a)$ implies $(b)$, where $(a)$ and $(b)$ are the two following statements.
\begin{itemize}
\item[\rm{(a)}]For any subsets $\mathcal{D}_{S_{1}},\mathcal{D}_{S_{2}}\subseteq\mathcal{D}_{[n]}\setminus\{n\}$ such that
\begin{itemize}
\item[\rm{(1)}]$d_{0}\in\mathcal{D}_{S_{1}}\cap\mathcal{D}_{S_{2}}$; and
\item[\rm{(2)}]$\spec(\ICG(n,\mathcal{D}_{S_{1}}))=\spec(\ICG(n,\mathcal{D}_{S_{2}}))$,
\end{itemize}
we have $\mathcal{D}_{S_{1}}=\mathcal{D}_{S_{2}}$.

\item[\rm{(b)}]For any subsets $\mathcal{D}_{S_{1}},\mathcal{D}_{S_{2}}\subseteq\mathcal{D}_{[n]}\setminus\{n\}$ such that
\begin{itemize}
\item[\rm{(1)}]$d_{0}\notin\mathcal{D}_{S_{1}}\triangle\mathcal{D}_{S_{2}}$; and
\item[\rm{(2)}]$\spec(\ICG(n,\mathcal{D}_{S_{1}}))=\spec(\ICG(n,\mathcal{D}_{S_{2}}))$,
\end{itemize}
we have $\mathcal{D}_{S_{1}}=\mathcal{D}_{S_{2}}$.
\end{itemize}
\end{cor}
\begin{proof}
Since $d_{0}\notin\mathcal{D}_{S_{1}}\triangle\mathcal{D}_{S_{2}}$, we have either $d_{0}\in\mathcal{D}_{S_{1}}\cap\mathcal{D}_{S_{2}}$ or $d_{0}\in\overline{\mathcal{D}_{S_{1}}}\cap\overline{\mathcal{D}_{S_{2}}}$. In the former case, by condition (a), we have $\mathcal{D}_{S_{1}}=\mathcal{D}_{S_{2}}$. In the latter case, by Lemma \ref{regulariso}, $\spec(\ICG(n,\overline{\mathcal{D}_{S_{1}}}))=\spec(\ICG(n,\overline{\mathcal{D}_{S_{2}}}))$. By condition (a), we have $\overline{\mathcal{D}_{S_{1}}}=\overline{\mathcal{D}_{S_{2}}}$ and so $\mathcal{D}_{S_{1}}=\mathcal{D}_{S_{2}}$.\qed

\end{proof}


\begin{lemma}\label{mcontradiction}
Let $n\geq1$ be an integer. Let $\ICG(n,\mathcal{D}_{S_{1}})$ and $\ICG(n,\mathcal{D}_{S_{2}})$ be isospectral integral circulant graphs. 
Let $\mathcal{D}_{R}$ be a subset of $\mathcal{D}_{[n]}$ such that $\forall d\in\mathcal{D}_{R}$, $\lambda_{d}(S_{1})=\lambda_{d}(S_{2})$. Let $\mathcal{D}_{T}\subseteq\mathcal{D}_{[n]}\setminus\mathcal{D}_{R}$ be a subset such that
\begin{itemize}
\item[\rm{(1)}]$\exists d_{0}\in\mathcal{D}_{T}$, s.t. $\forall d\in\mathcal{D}_{T}$, $\lambda_{d}(S_{1})=\lambda_{d_{0}}(S_{1})$; and
\item[\rm{(2)}]$\sum_{d\in\mathcal{D}_{T}}\phi(n/d)
>n-\sum_{d\in\mathcal{D}_{R}}\phi(n/d)-\sum_{d\in\mathcal{D}_{T}}\phi(n/d)$.
\end{itemize}
Then $\exists d_{0}'\in\mathcal{D}_{T}$, s.t. $\lambda_{d_{0}'}(S_{2})=\lambda_{d_{0}}(S_{1})$. In particular, if $\mathcal{D}_{T}=\{d_{0}\}$, then $\lambda_{d_{0}}(S_{1})=\lambda_{d_{0}}(S_{2})$.

\end{lemma}

\begin{proof}
Set
\begin{align*}
\spec(\ICG(n,\mathcal{D}_{S_{1}}))
=\spec(\ICG(n,\mathcal{D}_{S_{2}}))=\left(
\begin{array}{cccc}
\nu_{1} & \nu_{2} & \ldots & \nu_{J} \\
m_{1} & m_{2} & \ldots & m_{J}
\end{array}
\right).
\end{align*}
Note that $\forall d\in\mathcal{D}_{R}$, $\lambda_{d}(S_{1})=\lambda_{d}(S_{2})$. By Lemma \ref{sameorbiteigen}, $\forall k\in R$, $\lambda_{k}(S_{1})=\lambda_{k}(S_{2})$. By Lemma \ref{minusR}, $\forall j\in[J]$,
\begin{equation}\label{mcontradictioneq1}
|\mathcal{L}_{S_{1}}(\nu_{j})\setminus R|=|\mathcal{L}_{S_{2}}(\nu_{j})\setminus R|.
\end{equation}
Set $\nu_{j_{0}}=\lambda_{d_{0}}(S_{1})$. We proceed our proof by contradiction. Assume that $\forall d\in\mathcal{D}_{T}$, $\lambda_{d}(S_{2})\neq\lambda_{d_{0}}(S_{1})=\nu_{j_{0}}$. By Lemma \ref{sameorbiteigen}, $\forall k\in T$, $\lambda_{k}(S_{2})\neq\nu_{j_{0}}$. Then we have
\begin{align*}
|\mathcal{L}_{S_{2}}(\nu_{j_{0}})\setminus R|
&=|\mathcal{L}_{S_{2}}(\nu_{j_{0}})\setminus R\setminus T|
\\&=\sum_{d\in\mathcal{D}_{\mathcal{L}_{S_{2}}(\nu_{j_{0}})}\setminus\mathcal{D}_{R}\setminus\mathcal{D}_{T}}\phi(n/d)\tag{by (\ref{Scard})}
\\&\leq\sum_{d\in\mathcal{D}_{[n]}\setminus\mathcal{D}_{R}\setminus\mathcal{D}_{T}}\phi(n/d)
\\&=n-\sum_{d\in\mathcal{D}_{R}}\phi(n/d)-\sum_{d\in\mathcal{D}_{T}}\phi(n/d)
\\&<\sum_{d\in\mathcal{D}_{T}}\phi(n/d)\tag{by condition (2)}.
\end{align*}
Note that $\mathcal{D}_{T}\subseteq\mathcal{D}_{[n]}\setminus\mathcal{D}_{R}$, that is, $\mathcal{D}_{T}\cap\mathcal{D}_{R}=\emptyset$. By condition (1), $\mathcal{D}_{T}\subseteq\mathcal{D}_{\mathcal{L}_{S_{1}}(\nu_{j_{0}})}\setminus\mathcal{D}_{R}$. Then
\begin{align*}
\sum_{d\in\mathcal{D}_{T}}\phi(n/d)
\leq\sum_{d\in\mathcal{D}_{\mathcal{L}_{S_{1}}(\nu_{j_{0}})}\setminus\mathcal{D}_{R}}\phi(n/d)
=|\mathcal{L}_{S_{1}}(\nu_{j_{0}})\setminus R|.\tag{by (\ref{Scard})}
\end{align*}
Combining the above two inequalities, we have
$$
|\mathcal{L}_{S_{2}}(\nu_{j_{0}})\setminus R|<|\mathcal{L}_{S_{1}}(\nu_{j_{0}})\setminus R|,
$$
which contradicts (\ref{mcontradictioneq1}).\qed


\end{proof}





\begin{cor}\label{1equal}
Let $n\geq1$ be an integer such that $\phi(n)\geq n/2$. Let $\ICG(n,\mathcal{D}_{S_{1}})$ and $\ICG(n,\mathcal{D}_{S_{2}})$ be isospectral integral circulant graphs. Then $\lambda_{1}(S_{1})=\lambda_{1}(S_{2})$.

\end{cor}
\begin{proof}
By Lemma \ref{lambdan}, $\lambda_{n}(S_{1})=\lambda_{n}(S_{2})$. Since $\phi(n)\geq n/2$, we have
$$\phi(n/1)\geq n-\phi(n/1)>n-1-\phi(n/1)=n-\phi(n/n)-\phi(n/1).$$
Taking $\{n\}$, $\{1\}$, as $\mathcal{D}_{R}$, $\mathcal{D}_{T}$, in Lemma \ref{mcontradiction}, we have $\lambda_{1}(S_{1})=\lambda_{1}(S_{2})$.\qed

\end{proof}
\subsubsection{Proof of (c) of Theorem \ref{main}}
\begin{lemma}\label{2331}
Set $n=2^{3}3$. Let $\mathcal{D}_{S_{1}}$ and $\mathcal{D}_{S_{2}}$ be two subsets of $\mathcal{D}_{[n]}\setminus\{n\}$ such that $2^{3}\in\mathcal{D}_{S_{1}}\cap\mathcal{D}_{S_{2}}$. Then $\spec(\ICG(n,\mathcal{D}_{S_{1}}))
=\spec(\ICG(n,\mathcal{D}_{S_{2}}))$ implies $\mathcal{D}_{S_{1}}=\mathcal{D}_{S_{2}}$.
\end{lemma}

\begin{proof}
By Lemma \ref{lambdan}, $\lambda_{n}(S_{1})=\lambda_{n}(S_{2})$. By (\ref{ramanujan}) and (\ref{lambdaphimu}), $\sum_{d\in\mathcal{D}_{S_{1}}}\phi(n/d)=\sum_{d\in\mathcal{D}_{S_{2}}}\phi(n/d)$ and so
\begin{equation}\label{2331eq1}
\sum_{d\in\mathcal{D}_{S_{1}}\setminus\mathcal{D}_{S_{2}}}\phi(n/d)
=\sum_{d\in\mathcal{D}_{S_{2}}\setminus\mathcal{D}_{S_{1}}}\phi(n/d).
\end{equation}
By Lemmas \ref{n/2notintriangle} and \ref{n/pnotintriangle}, $2^{2}3,2^{3}\notin\mathcal{D}_{S_{1}}\triangle\mathcal{D}_{S_{2}}$ and so $\mathcal{D}_{S_{1}}\triangle\mathcal{D}_{S_{2}}\subseteq\{1,3,2,2\cdot3,2^{2}\}$. By (a) of Corollary \ref{n/2cor}, $\sum_{d\in\mathcal{D}_{S_{1}}\cap\{1,3\}}\phi(n/d)
=\sum_{d\in\mathcal{D}_{S_{2}}\cap\{1,3\}}\phi(n/d)$ and so
$$
\sum_{d\in(\mathcal{D}_{S_{1}}\setminus\mathcal{D}_{S_{2}})\cap\{1,3\}}\phi(n/d)
=\sum_{d\in(\mathcal{D}_{S_{2}}\setminus\mathcal{D}_{S_{1}})\cap\{1,3\}}\phi(n/d).
$$
Taking $\{1,3\}$, $\phi(n/\cdot)$, $(\mathcal{D}_{S_{1}}\setminus\mathcal{D}_{S_{2}})\cap\{1,3\}$, $(\mathcal{D}_{S_{2}}\setminus\mathcal{D}_{S_{1}})\cap\{1,3\}$, $\phi(n/\cdot)$, as $A$, $f$, $A_{1}$, $A_{2}$, in Corollary \ref{notincontradictioncor}, we have either $\phi(n/1)=\phi(n/3)$, which is impossible, or $(\mathcal{D}_{S_{1}}\triangle\mathcal{D}_{S_{2}})\cap\{1,3\}=(\mathcal{D}_{S_{1}}\cap\{1,3\})\triangle(\mathcal{D}_{S_{2}}\cap\{1,3\})=\emptyset$. Therefore,
\begin{equation}\label{2331eq4}
\mathcal{D}_{S_{1}}\triangle\mathcal{D}_{S_{2}}\subseteq\{2,2\cdot3,2^{2}\}
\end{equation}
To prove $\mathcal{D}_{S_{1}}=\mathcal{D}_{S_{2}}$, we rule out the following 3 cases.
\begin{itemize}
\item Case 1: $|\mathcal{D}_{S_{1}}\triangle\mathcal{D}_{S_{2}}|=3$\\
Then $\mathcal{D}_{S_{1}}\triangle\mathcal{D}_{S_{2}}=\{2,2\cdot3,2^{2}\}$. Note that $\phi(n/2)>\phi(n/(2\cdot3))=\phi(n/2^{2})$ and recall (\ref{2331eq1}). Taking $\{2,2\cdot3,2^{2}\}$, $\phi(n/\cdot)$, $\{2,2\cdot3\}$, $\mathcal{D}_{S_{1}}\setminus\mathcal{D}_{S_{2}}$, $\mathcal{D}_{S_{2}}\setminus\mathcal{D}_{S_{1}}$, as $A$, $f$, $B$, $A_{1}$, $A_{2}$, in Lemma \ref{notincontradiction}, we have
    \begin{equation}\label{2331eq2}
    \{2,2\cdot3\}\nsubseteq\mathcal{D}_{S_{1}}\setminus\mathcal{D}_{S_{2}}\quad\text{and}\quad \{2,2\cdot3\}\nsubseteq\mathcal{D}_{S_{2}}\setminus\mathcal{D}_{S_{1}}.
    \end{equation}
Taking $\{2,2\cdot3,2^{2}\}$, $\phi(n/\cdot)$, $\{2,2^{2}\}$, $\mathcal{D}_{S_{1}}\setminus\mathcal{D}_{S_{2}}$, $\mathcal{D}_{S_{2}}\setminus\mathcal{D}_{S_{1}}$, as $A$, $f$, $B$, $A_{1}$, $A_{2}$, in Lemma \ref{notincontradiction}, we have
    \begin{equation}\label{2331eq3}
    \{2,2^{2}\}\nsubseteq\mathcal{D}_{S_{1}}\setminus\mathcal{D}_{S_{2}}\quad\text{and}\quad \{2,2^{2}\}\nsubseteq\mathcal{D}_{S_{2}}\setminus\mathcal{D}_{S_{1}}.
    \end{equation}
With out loss of generality, we have
\begin{itemize}
\item Subcase 1.1: $\mathcal{D}_{S_{1}}\setminus\mathcal{D}_{S_{2}}=\{2,2\cdot3,2^{2}\}$ and $\mathcal{D}_{S_{2}}\setminus\mathcal{D}_{S_{1}}=\emptyset$\\
    This contradicts (\ref{2331eq2}) and (\ref{2331eq3}).
\item Subcase 1.2: $\mathcal{D}_{S_{1}}\setminus\mathcal{D}_{S_{2}}=\{2,2\cdot3\}$ and $\mathcal{D}_{S_{2}}\setminus\mathcal{D}_{S_{1}}=\{2^{2}\}$\\
    This contradicts (\ref{2331eq2}).
\item Subcase 1.3: $\mathcal{D}_{S_{1}}\setminus\mathcal{D}_{S_{2}}=\{2,2^{2}\}$ and $\mathcal{D}_{S_{2}}\setminus\mathcal{D}_{S_{1}}=\{2\cdot3\}$\\
    This contradicts (\ref{2331eq3}).
\item Subcase 1.4: $\mathcal{D}_{S_{1}}\setminus\mathcal{D}_{S_{2}}=\{2\}$ and $\mathcal{D}_{S_{2}}\setminus\mathcal{D}_{S_{1}}=\{2\cdot3,2^{2}\}$\\
    Recall $2^{3}\in\mathcal{D}_{S_{1}}\cap\mathcal{D}_{S_{2}}$. By Table \ref{table2331a}, $\spec(\ICG(n,\mathcal{D}_{S_{1}}))\neq\spec(\ICG(n,\mathcal{D}_{S_{2}}))$, which is a contradiction.

\end{itemize}
\end{itemize}
    \begin{table}[!h]
    \centering
{\tiny
    \caption{$n=2^{3}3$, $\mathcal{D}_{S_{1}}\setminus\mathcal{D}_{S_{2}}=\{2\}$, $\mathcal{D}_{S_{2}}\setminus\mathcal{D}_{S_{1}}=\{2\cdot3,2^{2}\}$ and $2^{3}\in\mathcal{D}_{S_{1}}\cap\mathcal{D}_{S_{2}}$\label{table2331a}}
    \begin{tabular}{|c|c|c|}
    \hline
    $\mathcal{D}_{S_{1}}\cap\mathcal{D}_{S_{2}}$ &   $\spec(\ICG(n,\mathcal{D}_{S_{1}}))$    &
    $\spec(\ICG(n,\mathcal{D}_{S_{2}}))$  \\\hline

 $\{2^{3}\}$

    &

$\left(
\begin{array}{cccccc}
6 & 2 & 1 & -1 & -2 &-3 \\
2 & 4 & 4 & 8 & 2 & 4
\end{array}
\right)$

    &

$\left(
\begin{array}{cccc}
6 & 2 & 0 & -4\\
2 & 2 & 16 & 4
\end{array}
\right)$\\\hline

 $\{2^{2}\cdot3,2^{3}\}$

    &

$\left(
\begin{array}{ccccc}
7 & 2 & 1 & -1 & -2 \\
2 & 4 & 4 & 2 & 12
\end{array}
\right)$

    &

$\left(
\begin{array}{ccccc}
7 & 3 & 1 & -1 & -3\\
2 & 2 & 4 & 12 &4
\end{array}
\right)$\\\hline

 $\{3,2^{3}\}$

    &

$\left(
\begin{array}{cccccc}
10 & 2 & 1 & -1 & -2 & -7 \\
1 & 5 & 6 & 8 & 2 & 2
\end{array}
\right)$

    &

$\left(
\begin{array}{ccccc}
10 & 4 & 2 & 0 & -4\\
1 & 2 & 3 & 12 &6
\end{array}
\right)$\\\hline

 $\{3,2^{2}\cdot3,2^{3}\}$

    &

$\left(
\begin{array}{ccccccc}
11 & 3 & 2 & 1 & -1 & -2 & -6 \\
1 & 1 & 6 & 4 & 2 & 8 & 2
\end{array}
\right)$

    &

$\left(
\begin{array}{ccccc}
11 & 5 & 3 & -1 & -3\\
1 & 2 & 3 & 12 &6
\end{array}
\right)$\\\hline

 $\{1,2^{3}\}$

    &

$\left(
\begin{array}{cccccc}
14 & 2 & 1 & -1 & -2 & -7 \\
1 & 4 & 6 & 8 & 3 & 2
\end{array}
\right)$

    &

$\left(
\begin{array}{cccccc}
14 & 4 & 2 & 0 & -2 & -4\\
1 & 2 & 2 & 12 & 1& 6
\end{array}
\right)$\\\hline

 $\{1,2^{2}\cdot3,2^{3}\}$

    &

$\left(
\begin{array}{cccccc}
15 & 2 & 1 & -1 & -2 & -6 \\
1 & 6 & 4 & 3 & 8 & 2
\end{array}
\right)$

    &

$\left(
\begin{array}{ccccc}
15 & 5 & 3 & -1 & -3\\
1 & 2 & 2 & 13 &6
\end{array}
\right)$\\\hline

 $\{1,3,2^{3}\}$

    &

$\left(
\begin{array}{ccccccc}
18 & 2 & 1 & -1 & -2 & -3 & -6 \\
1 & 4 & 4 & 8 & 2 & 4 &1
\end{array}
\right)$

    &

$\left(
\begin{array}{ccccc}
18 & 2 & 0 & -4 & -6\\
1 & 2 & 16 & 4 & 1
\end{array}
\right)$\\\hline

 $\{1,3,2^{2}\cdot3,2^{3}\}$

    &

$\left(
\begin{array}{cccccc}
19 & 2 & 1 & -1 & -2 & -5 \\
1 & 4 & 4 & 2 & 12 & 1
\end{array}
\right)$

    &

$\left(
\begin{array}{cccccc}
19 & 3 & 1 & -1 & -3 &-5\\
1 & 2 & 4 & 12 &4 &1
\end{array}
\right)$\\\hline

    \end{tabular}

 }
 \end{table}
\begin{itemize}
\item Case 2: $|\mathcal{D}_{S_{1}}\triangle\mathcal{D}_{S_{2}}|=2$\\
Then $\mathcal{D}_{S_{1}}\triangle\mathcal{D}_{S_{2}}=\{2,2\cdot3\}$, $\{2,2^{2}\}$ or $\{2\cdot3,2^{2}\}$. Set $\mathcal{D}_{S_{1}}\triangle\mathcal{D}_{S_{2}}=\{a_{1},a_{2}\}$. Taking $\mathcal{D}_{S_{1}}\triangle\mathcal{D}_{S_{2}}$, $\phi(n/\cdot)$, $\mathcal{D}_{S_{1}}\setminus\mathcal{D}_{S_{2}}$, $\mathcal{D}_{S_{2}}\setminus\mathcal{D}_{S_{1}}$, as $A$, $f$, $A_{1}$, $A_{2}$, in Corollary \ref{notincontradictioncor}, we have either $\phi(n/a_{1})=\phi(n/a_{2})$ while $|\mathcal{D}_{S_{1}}\setminus\mathcal{D}_{S_{2}}|=|\mathcal{D}_{S_{2}}\setminus\mathcal{D}_{S_{1}}|=1$, or $\mathcal{D}_{S_{1}}\triangle\mathcal{D}_{S_{2}}=\emptyset$ leading to a contradiction. Note that $\phi(n/2)\neq\phi(n/(2\cdot3))$, that $\phi(n/2)\neq\phi(n/(2^{2}))$, and that $\phi(n/(2\cdot3))=\phi(n/2^{2})$. Without loss of generality, we have $\mathcal{D}_{S_{1}}\setminus\mathcal{D}_{S_{2}}=\{2\cdot3\}$ and $\mathcal{D}_{S_{2}}\setminus\mathcal{D}_{S_{1}}=\{2^{2}\}$. Recall $2^{3}\in\mathcal{D}_{S_{1}}\cap\mathcal{D}_{S_{2}}$. By Table \ref{table2331b}, $\spec(\ICG(n,\mathcal{D}_{S_{1}}))\neq\spec(\ICG(n,\mathcal{D}_{S_{2}}))$, which is a contradiction.

\item Case 3: $|\mathcal{D}_{S_{1}}\triangle\mathcal{D}_{S_{2}}|=1$\\
Recall (\ref{2331eq1}). Taking $\mathcal{D}_{S_{1}}\triangle\mathcal{D}_{S_{2}}$, $\phi(n/\cdot)$, $\mathcal{D}_{S_{1}}\setminus\mathcal{D}_{S_{2}}$, $\mathcal{D}_{S_{2}}\setminus\mathcal{D}_{S_{1}}$, as $A$, $f$, $A_{1}$, $A_{2}$, in Lemma \ref{notincontradictionnew}, we have $\mathcal{D}_{S_{1}}\triangle\mathcal{D}_{S_{2}}=\emptyset$, which is a contradiction.


\end{itemize}
\begin{table}[!h]
\centering
{\tiny

\caption{$n=2^{3}3$, $\mathcal{D}_{S_{1}}\setminus\mathcal{D}_{S_{2}}=\{2\cdot3\}$, $\mathcal{D}_{S_{2}}\setminus\mathcal{D}_{S_{1}}=\{2^{2}\}$ and $2^{3}\in\mathcal{D}_{S_{1}}\cap\mathcal{D}_{S_{2}}$\label{table2331b}}
\begin{tabular}{|c|c|c|}\hline
$\mathcal{D}_{S_{1}}\cap\mathcal{D}_{S_{2}}$ &   $\spec(\ICG(n,\mathcal{D}_{S_{1}}))$    &
    $\spec(\ICG(n,\mathcal{D}_{S_{2}}))$  \\\hline

 $\{2^{3}\}$

    &

$\left(
\begin{array}{cccccc}
4 & 2 & 1 & 0 & -1 & -3 \\
2 & 4 & 4 & 2 & 8 & 4
\end{array}
\right)$

    &

$\left(
\begin{array}{ccc}
4& 0 & -2\\
4 & 12 & 8
\end{array}
\right)$\\\hline

 $\{2^{2}\cdot3,2^{3}\}$

    &

$\left(
\begin{array}{cccc}
5 & 2 & 1 &-2 \\
2 & 4 & 6 & 12
\end{array}
\right)$

    &

$\left(
\begin{array}{cc}
5 & -1\\
4 & 20
\end{array}
\right)$\\\hline

     $\{2,2^{3}\}$

    &

$\left(
\begin{array}{cccc}
8 & 2 & -1 &-4 \\
2 & 4 & 16 & 2
\end{array}
\right)$

    &

$\left(
\begin{array}{ccc}
8 & 0 & -4\\
2 & 18 & 4
\end{array}
\right)$\\\hline

$\{2,2^{2}\cdot3,2^{3}\}$

    &

$\left(
\begin{array}{ccccc}
9 & 1 & 0 &-2  &-3\\
2 & 4 & 8 & 8 &2
\end{array}
\right)$

    &

$\left(
\begin{array}{cccc}
9 & 1 & -1 & -3\\
2 & 6 & 12 & 4
\end{array}
\right)$\\\hline

     $\{3,2^{3}\}$

    &

$\left(
\begin{array}{cccccc}
8& 5 & 2 & 0 & -1 & -3\\
1 & 2 & 4 & 3 & 8 & 6
\end{array}
\right)$

    &

$\left(
\begin{array}{cccccc}
8& 4 & 2 & 0 & -2 & -6\\
1 & 2 & 2 & 13 & 4 &2
\end{array}
\right)$\\\hline

  $\{3, 2^{2}\cdot3,2^{3}\}$

    &

$\left(
\begin{array}{cccc}
9 & 6 & 1 &-2 \\
1 & 2 & 7 & 14
\end{array}
\right)$

    &

$\left(
\begin{array}{cccccc}
9 & 5 & 3 & 1 & -1 &-5\\
1 & 2 & 2 & 1 & 16 & 2
\end{array}
\right)$\\\hline

    $\{3,2,2^{3}\}$

    &

$\left(
\begin{array}{ccccccc}
12 & 4 & 3 & 2 &-1 &-4 & -5\\
1 & 1 & 2& 4 & 12 & 2& 2
\end{array}
\right)$

    &

$\left(
\begin{array}{cccc}
12 & 4 & 0 & -8\\
1& 1 & 20 & 2
\end{array}
\right)$\\\hline

      $\{3,2,2^{2}\cdot3,2^{3}\}$

    &

$\left(
\begin{array}{cccccccc}
13& 5 & 4 & 1 & 0 &-2 & -3 & -4\\
1& 1 & 2 & 4 & 4 & 8 &2 & 2
\end{array}
\right)$

    &

$\left(
\begin{array}{ccccc}
13 & 5 & 1 & -1 & -7\\
1 & 1 & 8& 12 & 2
\end{array}
\right)$\\\hline

      $\{1,2^{3}\}$

    &

$\left(
\begin{array}{ccccccc}
12 & 5 & 2 & 0 & -1 & -3 & -4\\
1 & 2 & 4 & 2& 8 & 6 & 1
\end{array}
\right)$

    &

$\left(
\begin{array}{ccccccc}
12 & 4 & 2 & 0 & -2 & -4 & -6\\
1& 2 & 2 &12 & 4 & 1 & 2
\end{array}
\right)$\\\hline

        $\{1,2^{2}\cdot3,2^{3}\}$

    &

$\left(
\begin{array}{ccccc}
13 & 6 & 1 & -2 & -3\\
1 & 2& 6 & 14 & 1
\end{array}
\right)$

    &

$\left(
\begin{array}{cccccc}
13 & 5 & 3 & -1 & -3 & -5\\
1 & 2 & 2 & 16 & 1 & 2
\end{array}
\right)$\\\hline

        $\{1,2,2^{3}\}$

    &

$\left(
\begin{array}{ccccccc}
16 & 3 & 2 & 0 & -1 & -4 & -5\\
1 & 2 & 4 & 1 & 12 & 2& 2
\end{array}
\right)$

    &

$\left(
\begin{array}{ccc}
16 & 0 & -8\\
1 & 21 & 2
\end{array}
\right)$\\\hline

          $\{1,2,2^{2}\cdot3,2^{3}\}$

    &

$\left(
\begin{array}{ccccccc}
17 & 4 & 1 & 0 & -2 & -3 & -4\\
1 &2 & 5 & 4 & 8 & 2 & 2
\end{array}
\right)$

    &

$\left(
\begin{array}{cccc}
17 & 1 & -1 & -7\\
1 & 9 & 12 & 2
\end{array}
\right)$\\\hline

          $\{1,3,2^{3}\}$

    &

$\left(
\begin{array}{ccccccc}
16 & 2 & 1 & 0 & -1 & -3 & -8\\
1 & 4 & 4 & 2 & 8 & 4& 1
\end{array}
\right)$

    &

$\left(
\begin{array}{ccccc}
16 & 4 & 0 & -2 & -8\\
1 & 2 & 12 & 8 & 1
\end{array}
\right)$\\\hline

             $\{1,3,2^{2}\cdot3,2^{3}\}$

    &

$\left(
\begin{array}{ccccc}
17& 2 & 1 & -2 & -7\\
1& 4& 6 & 12 &1
\end{array}
\right)$

    &

$\left(
\begin{array}{cccc}
17& 5 & -1 & -7\\
1 & 2 & 20 & 1
\end{array}
\right)$\\\hline

             $\{1,3,2,2^{3}\}$

    &

$\left(
\begin{array}{cccc}
20 & 2 & -1 & -4\\
1 & 4 & 16 & 3
\end{array}
\right)$

    &

$\left(
\begin{array}{ccc}
20 & 0 & -4\\
1 & 18 & 5
\end{array}
\right)$\\\hline

             $\{1,3,2,2^{2}\cdot3,2^{3}\}$

    &

$\left(
\begin{array}{ccccc}
21 & 1 & 0 & -2 & -3\\
1 & 4 & 8 & 8 & 3
\end{array}
\right)$

    &

$\left(
\begin{array}{cccc}
21 & 1 & -1 & -3\\
1 & 6 & 12 & 5
\end{array}
\right)$\\\hline

        \end{tabular}
}
        \end{table}
This completes the proof.\qed



\end{proof}

\begin{theorem}\label{233}
Set $n=2^{3}3$. Let $\mathcal{D}_{S_{1}}$ and $\mathcal{D}_{S_{2}}$ be two subsets of $\mathcal{D}_{[n]}\setminus\{n\}$. Then $\spec(\ICG(n,\mathcal{D}_{S_{1}}))
=\spec(\ICG(n,\mathcal{D}_{S_{2}}))$ implies $\mathcal{D}_{S_{1}}=\mathcal{D}_{S_{2}}$.

\end{theorem}
\begin{proof}
By Lemma \ref{n/pnotintriangle}, $2^{3}\notin\mathcal{D}_{S_{1}}\triangle\mathcal{D}_{S_{2}}$. By Lemma \ref{2331} and Corollary \ref{regularisocor}, we obtain the result.\qed

\end{proof}




\begin{theorem}\label{23q}
Set $n=2^{3}q$ with prime $q>3$. Let $\mathcal{D}_{S_{1}}$ and $\mathcal{D}_{S_{2}}$ be two subsets of $\mathcal{D}_{[n]}\setminus\{n\}$. Then $\spec(\ICG(n,\mathcal{D}_{S_{1}}))=\spec(\ICG(n,\mathcal{D}_{S_{2}}))$ implies $\mathcal{D}_{S_{1}}=\mathcal{D}_{S_{2}}$.

\end{theorem}

\begin{proof}
By Lemma \ref{lambdan}, $\lambda_{n}(S_{1})=\lambda_{n}(S_{2})$. By (\ref{ramanujan}) and (\ref{lambdaphimu}), $\sum_{d\in\mathcal{D}_{S_{1}}}\phi(n/d)
=\sum_{d\in\mathcal{D}_{S_{2}}}\phi(n/d)$ and so
$$
\sum_{d\in\mathcal{D}_{S_{1}}\setminus\mathcal{D}_{S_{2}}}\phi(n/d)
=\sum_{d\in\mathcal{D}_{S_{2}}\setminus\mathcal{D}_{S_{1}}}\phi(n/d).
$$
Similar to (\ref{2331eq4}) in the proof of Lemma \ref{2331}, replacing $3$ by $q$, we have $\mathcal{D}_{S_{1}}\triangle\mathcal{D}_{S_{2}}\subseteq\{2,2q,2^{2}\}$. Therefore,
$$
\sum_{d\in\{2,2^{2},2q\}}\chi_{\mathcal{D}_{S_{1}}}(d)\phi(n/d)
=\sum_{d\in\{2,2^{2},2q\}}\chi_{\mathcal{D}_{S_{2}}}(d)\phi(n/d),
$$
which is the condition (2) of Lemma \ref{supersequence1}. Set $(x_{0},x_{1},x_{2})=(\phi(n/2q),\phi(n/2^{2}),\phi(n/2))$, which is a $1$-super sequence. Set $(a_{0},a_{1},a_{2})=(\chi_{\mathcal{D}_{S_{1}}}(2q),\chi_{\mathcal{D}_{S_{1}}}(2^{2}),\chi_{\mathcal{D}_{S_{1}}}(2))$, $(b_{0},b_{1},b_{2})=(\chi_{\mathcal{D}_{S_{2}}}(2q),\chi_{\mathcal{D}_{S_{2}}}(2^{2}),\chi_{\mathcal{D}_{S_{2}}}(2))$, which clearly satisfies the condition (1) of Lemma \ref{supersequence1}. By Lemma \ref{supersequence1}, $(a_{0},a_{1},a_{2})=(b_{0},b_{1},b_{2})$, that is, $\mathcal{D}_{S_{1}}\cap\{2,2^{2},2q\}=\mathcal{D}_{S_{2}}\cap\{2,2^{2},2q\}$. Thus, $\mathcal{D}_{S_{1}}=\mathcal{D}_{S_{2}}$. This completes the proof.\qed

\end{proof}

\begin{lemma}\label{p3q1}
Set $n=p^{3}q$ with primes $3\leq p<q$. Let $\mathcal{D}_{S_{1}}$ and $\mathcal{D}_{S_{2}}$ be two subsets of $\mathcal{D}_{[n]}\setminus\{n\}$ such that $1\in\mathcal{D}_{S_{1}}\cap\mathcal{D}_{S_{2}}$. Then $\spec(\ICG(n,\mathcal{D}_{S_{1}}))
=\spec(\ICG(n,\mathcal{D}_{S_{2}}))$ implies $\mathcal{D}_{S_{1}}=\mathcal{D}_{S_{2}}$.

\end{lemma}

\begin{proof}
By Lemma \ref{lambdan}, $\lambda_{n}(S_{1})=\lambda_{n}(S_{2})$. By (\ref{ramanujan}) and (\ref{lambdaphimu}), $\sum_{d\in\mathcal{D}_{S_{1}}}\phi(n/d)
=\sum_{d\in\mathcal{D}_{S_{2}}}\phi(n/d)$ and so
\begin{equation}\label{p3q1eq2}
\sum_{d\in\mathcal{D}_{S_{1}}\setminus\mathcal{D}_{S_{2}}}\phi(n/d)
=\sum_{d\in\mathcal{D}_{S_{2}}\setminus\mathcal{D}_{S_{1}}}\phi(n/d).
\end{equation}
By Lemma \ref{n/pnotintriangle}, $p^{2}q,p^{3}\notin\mathcal{D}_{S_{1}}\triangle\mathcal{D}_{S_{2}}$ and so $\mathcal{D}_{S_{1}}\triangle\mathcal{D}_{S_{2}}\subseteq\{1,q,p,pq,p^{2}\}$. Since $1\in\mathcal{D}_{S_{1}}\cap\mathcal{D}_{S_{2}}$ and so $1\notin\mathcal{D}_{S_{1}}\triangle\mathcal{D}_{S_{2}}$, we have $\mathcal{D}_{S_{1}}\triangle\mathcal{D}_{S_{2}}\subseteq\{q,p,pq,p^{2}\}$. Besides, by Corollary \ref{1equal}, $\lambda_{1}(S_{1})=\lambda_{1}(S_{2})$. Hence, by $(\ref{ramanujan})$ and $(\ref{lambdaphimu})$, $\sum_{d\in\mathcal{D}_{S_{1}}}\mu(n/d)=\sum_{d\in\mathcal{D}_{S_{2}}}\mu(n/d)$ and so
$$
\sum_{d\in\mathcal{D}_{S_{1}}\setminus\mathcal{D}_{S_{2}}}\mu(n/d)
=\sum_{d\in\mathcal{D}_{S_{2}}\setminus\mathcal{D}_{S_{1}}}\mu(n/d).
$$
Note that $\mu(n/p^{2})=1$ and that $\forall d\in\{q,p,pq\}$, $\mu(n/d)=0$. Taking $\{q,p,pq,p^{2}\}$, $\mu(n/\cdot)$, $\{p^{2}\}$, $\mathcal{D}_{S_{1}}\setminus\mathcal{D}_{S_{2}}$, $\mathcal{D}_{S_{2}}\setminus\mathcal{D}_{S_{1}}$, as $A$, $f$, $B$, $A_{1}$, $A_{2}$, in Lemma \ref{notincontradiction}, we have $p^{2}\notin\mathcal{D}_{S_{1}}\triangle\mathcal{D}_{S_{2}}$ and so
$$\mathcal{D}_{S_{1}}\triangle\mathcal{D}_{S_{2}}\subseteq\{q,p,pq\}.$$
To prove $\mathcal{D}_{S_{1}}=\mathcal{D}_{S_{2}}$, we rule out the following 3 cases.
\begin{itemize}
\item Case 1: $|\mathcal{D}_{S_{1}}\triangle\mathcal{D}_{S_{2}}|=3$\\
Then $\mathcal{D}_{S_{1}}\triangle\mathcal{D}_{S_{2}}=\{q,p,pq\}$. Note that $\phi(n/p)>\phi(n/q)>\phi(n/pq)$ and recall (\ref{p3q1eq2}). Similar to Case 1 in the proof of Lemma \ref{2331}, replacing $2,2\cdot3,2^{2}$ by $p,q,pq$, without loss of generality, we only need to rule out the subcase where $\mathcal{D}_{S_{1}}\setminus\mathcal{D}_{S_{2}}=\{p\}$ and $\mathcal{D}_{S_{2}}\setminus\mathcal{D}_{S_{1}}=\{q,pq\}$. In this subcase, (\ref{p3q1eq2}) implies that
\begin{equation}\label{p3q1eq1}
q=p+2.
\end{equation}
Recalling that $1\in\mathcal{D}_{S_{1}}\cap\mathcal{D}_{S_{2}}$, we have $\{1,p\}\subseteq\mathcal{D}_{S_{1}}\subseteq\{1,p,p^{2},p^{3},p^{2}q\}$. Therefore,
\begin{table}[!h]\centering
{\tiny
\caption{Some Ramanujan sums when $n=p^{3}q$\label{ramanujanp3q}}
\begin{tabular}{|c|c|c|c|c|c|c|c|}
\hline
$d$ takes (values) & $1$ & $p$ & $p^{2}$ & $p^{3}$ & $q$ & $pq$ & $p^{2}q$\\
\hline
$n/d$ equals to &$p^{3}q$ & $p^{2}q$ & $pq$ &$q$ & $p^{3}$ & $p^{2}$& $p$\\
\hline
$\mathcal{R}_{n/d}(p)$ equals to &$0$ & $p$ & $-(p-1)$ & $-1$ & $0$ & $-p$ &$p-1$\\
\hline
$\mathcal{R}_{n/d}(p^{2})$ equals to &$p^{2}$ & $-p(p-1)$ &$-(p-1)$ &$-1$ &$-p^{2}$ &$p(p-1)$ &$p-1$\\
\hline
\end{tabular}
}
\end{table}
\begin{align*}
\lambda_{p}(S_{1})
&=\mathcal{R}_{n/1}(p)+\mathcal{R}_{n/p}(p)
+\sum_{d\in\mathcal{D}_{S_{1}}\cap\{p^{2},p^{3},p^{2}q\}}\mathcal{R}_{n/d}(p)\tag{by (\ref{ramanujan})}
\\&=\mathcal{R}_{n/1}(p^{2})+\mathcal{R}_{n/p}(p^{2})
+\sum_{d\in\mathcal{D}_{S_{1}}\cap\{p^{2},p^{3},p^{2}q\}}\mathcal{R}_{n/d}(p^{2})\tag{by Table \ref{ramanujanp3q}}
\\&=\lambda_{p^{2}}(S_{1})\tag{by (\ref{ramanujan})}.
\end{align*}
Since $p\geq3$, we have $p^{3}+p^{2}-p-1>p^{3}+p$, which by (\ref{p3q1eq1}), implies
$$
\sum_{d\in\{p,p^{2}\}}\phi(n/d)>n-\sum_{d\in\{1,p^{3}q\}}\phi(n/d)-\sum_{d\in\{p,p^{2}\}}\phi(n/d).
$$
Taking $\{1,p^{3}q\}$, $\{p,p^{2}\}$, as $\mathcal{D}_{R}$, $\mathcal{D}_{T}$, in Lemma \ref{mcontradiction}, we have either
$$\lambda_{p}(S_{1})=\lambda_{p}(S_{2})\qquad\mathrm{or}\qquad \lambda_{p^{2}}(S_{1})=\lambda_{p^{2}}(S_{2}).$$
Note that $\mathcal{D}_{S_{1}}=(\mathcal{D}_{S_{1}}\setminus\mathcal{D}_{S_{2}})\cup(\mathcal{D}_{S_{1}}\cap\mathcal{D}_{S_{2}})$ and that $\mathcal{D}_{S_{2}}=(\mathcal{D}_{S_{2}}\setminus\mathcal{D}_{S_{1}})\cup(\mathcal{D}_{S_{1}}\cap\mathcal{D}_{S_{2}})$. We have
\begin{align*}
\lambda_{p}(S_{1})&=\mathcal{R}_{n/p}(p)
+\sum_{d\in\mathcal{D}_{S_{1}}\cap\mathcal{D}_{S_{2}}}\mathcal{R}_{n/d}(p)\tag{by (\ref{ramanujan})}
\\&>\mathcal{R}_{n/q}(p)+\mathcal{R}_{n/pq}(p)
+\sum_{d\in\mathcal{D}_{S_{1}}\cap\mathcal{D}_{S_{2}}}\mathcal{R}_{n/d}(p)\tag{by Table \ref{ramanujanp3q}}
\\&=\lambda_{p}(S_{2})\tag{by (\ref{ramanujan})}
\end{align*}
and
\begin{align*}
\lambda_{p^{2}}(S_{1})&=\mathcal{R}_{n/p}(p^{2})
+\sum_{d\in\mathcal{D}_{S_{1}}\cap\mathcal{D}_{S_{2}}}\mathcal{R}_{n/d}(p^{2})\tag{by (\ref{ramanujan})}
\\&<\mathcal{R}_{n/q}(p^{2})+\mathcal{R}_{n/pq}(p^{2})
+\sum_{d\in\mathcal{D}_{S_{1}}\cap\mathcal{D}_{S_{2}}}\mathcal{R}_{n/d}(p^{2})\tag{by Table \ref{ramanujanp3q}}
\\&=\lambda_{p^{2}}(S_{2})\tag{by (\ref{ramanujan})},
\end{align*}
which is a contradiction.

\item Case 2: $|\mathcal{D}_{S_{1}}\triangle\mathcal{D}_{S_{2}}|=2$\\
Then $\mathcal{D}_{S_{1}}\triangle\mathcal{D}_{S_{2}}=\{p,pq\}$, $\{p,p^{2}\}$, or $\{pq,p^{2}\}$. Recall (\ref{p3q1eq2}). Note that $\phi(n/p)>\phi(n/q)>\phi(n/pq)$. Taking $\mathcal{D}_{S_{1}}\triangle\mathcal{D}_{S_{2}}$, $\phi(n/\cdot)$, $\mathcal{D}_{S_{1}}\setminus\mathcal{D}_{S_{2}}$, $\mathcal{D}_{S_{2}}\setminus\mathcal{D}_{S_{1}}$, as $A$, $f$, $A_{1}$, $A_{2}$, in Corollary \ref{notincontradictioncor}, we have $\mathcal{D}_{S_{1}}\triangle\mathcal{D}_{S_{2}}=\emptyset$, which is a contradiction.
\item Case 3: $|\mathcal{D}_{S_{1}}\triangle\mathcal{D}_{S_{2}}|=1$\\
Recall (\ref{p3q1eq2}). Taking $\mathcal{D}_{S_{1}}\triangle\mathcal{D}_{S_{2}}$, $\phi(n/\cdot)$, $\mathcal{D}_{S_{1}}\setminus\mathcal{D}_{S_{2}}$, $\mathcal{D}_{S_{2}}\setminus\mathcal{D}_{S_{1}}$, as $A$, $f$, $A_{1}$, $A_{2}$, in Lemma \ref{notincontradictionnew}, we have $\mathcal{D}_{S_{1}}\triangle\mathcal{D}_{S_{2}}=\emptyset$, which is a contradiction.

\end{itemize}
This completes the proof.\qed
\end{proof}
\begin{theorem}\label{p3q}
Set $n=p^{3}q$ with primes $3\leq p<q$. Let $\mathcal{D}_{S_{1}}$ and $\mathcal{D}_{S_{2}}$ be two subsets of $\mathcal{D}_{[n]}\setminus\{n\}$. Then $\spec(\ICG(n,\mathcal{D}_{S_{1}}))
=\spec(\ICG(n,\mathcal{D}_{S_{2}}))$ implies $\mathcal{D}_{S_{1}}=\mathcal{D}_{S_{2}}$.

\end{theorem}
\begin{proof}
Since $\phi(n)/n=\frac{(p-1)(q-1)}{pq}\geq\frac{2\cdot4}{3\cdot5}>\frac{1}{2}$, we have $\phi(n)>n/2$. By Lemma \ref{1notintriangle}, $1\notin\mathcal{D}_{S_{1}}\triangle\mathcal{D}_{S_{2}}$. By Lemma \ref{p3q} and Corollary \ref{regularisocor}, we obtain the result.\qed

\end{proof}

\subsubsection{Proof of (d) of Theorem \ref{main}}
\begin{lemma}\label{22321}
Set $n=2^{2}3^{2}$. Let $\mathcal{D}_{S_{1}}$ and $\mathcal{D}_{S_{2}}$ be two subsets of $\mathcal{D}_{[n]}\setminus\{n\}$ such that $2^{2}3\in\mathcal{D}_{S_{1}}\cap\mathcal{D}_{S_{2}}$. Then $\spec(\ICG(n,\mathcal{D}_{S_{1}}))
=\spec(\ICG(n,\mathcal{D}_{S_{2}}))$ implies $\mathcal{D}_{S_{1}}=\mathcal{D}_{S_{2}}$.
\end{lemma}
\begin{proof}
By Lemma \ref{n/2notintriangle} and \ref{n/pnotintriangle}, $2\cdot3^{2},2^{2}3\notin\mathcal{D}_{S_{1}}\triangle\mathcal{D}_{S_{2}}$ and so $\mathcal{D}_{S_{1}}\triangle\mathcal{D}_{S_{2}}\subseteq\{1,3,3^{2},2,2\cdot3,2^{2}\}$. By Lemma \ref{lambdan}, $\lambda_{n}(S_{1})=\lambda_{n}(S_{2})$. By (\ref{ramanujan}) and (\ref{lambdaphimu}), we have $\sum_{d\in\mathcal{D}_{S_{1}}}\phi(n/d)
=\sum_{d\in\mathcal{D}_{S_{2}}}\phi(n/d)$ and so
\begin{equation}\label{22321eq1}
\sum_{d\in\mathcal{D}_{S_{1}}\setminus\mathcal{D}_{S_{2}}}\phi(n/d)
=\sum_{d\in\mathcal{D}_{S_{2}}\setminus\mathcal{D}_{S_{1}}}\phi(n/d).
\end{equation}

We first use the method in proof of Lemma \ref{23q} to prove that $\mathcal{D}_{S_{1}}\cap\{1,3,3^{2}\}=\mathcal{D}_{S_{2}}\cap\{1,3,3^{2}\}$. By (a) of Corollary \ref{n/2cor},
$$
\sum_{d\in\{1,3,3^{2}\}}\chi_{\mathcal{D}_{S_{1}}}(d)\phi(n/d)
=\sum_{d\in\{1,3,3^{2}\}}\chi_{\mathcal{D}_{S_{2}}}(d)\phi(n/d),
$$
which is the condition (2) of Lemma \ref{supersequence1}. Set $(x_{0},x_{1},x_{2})=(\phi(n/3^{2}),\phi(n/3),\phi(n/1))$, which is a $1$-super sequence. Set $(a_{0},a_{1},a_{2})
=(\chi_{\mathcal{D}_{S_{1}}}(3^{2}),\chi_{\mathcal{D}_{S_{1}}}(3),\chi_{\mathcal{D}_{S_{1}}}(1))$ and correspondingly, $(b_{0},b_{1},b_{2})
=(\chi_{\mathcal{D}_{S_{2}}}(3^{2}),\chi_{\mathcal{D}_{S_{2}}}(3),\chi_{\mathcal{D}_{S_{2}}}(1))$, which clearly satisfies the condition (1) of Lemma \ref{supersequence1}. By Lemma \ref{supersequence1}, $\mathcal{D}_{S_{1}}\cap\{1,3,3^{2}\}=\mathcal{D}_{S_{2}}\cap\{1,3,3^{2}\}$.
So
\begin{equation}\label{22321eq5}
\mathcal{D}_{S_{1}}\triangle\mathcal{D}_{S_{2}}\subseteq\{2,2\cdot3,2^{2}\}.
\end{equation}
To prove $\mathcal{D}_{S_{1}}=\mathcal{D}_{S_{2}}$, we rule out the following 3 cases.
\begin{itemize}
\item Case 1: $|\mathcal{D}_{S_{1}}\triangle\mathcal{D}_{S_{2}}|=3$\\
Then $\mathcal{D}_{S_{1}}\triangle\mathcal{D}_{S_{2}}=\{2,2\cdot3,2^{2}\}$. Note that $\phi(n/2)=\phi(n/2^{2})>\phi(n/(2\cdot3))$ and recall (\ref{22321eq1}). Taking $\{2,2\cdot3,2^{2}\}$, $\phi(n/\cdot)$, $\{2,2\cdot3\}$, $\mathcal{D}_{S_{1}}\setminus\mathcal{D}_{S_{2}}$, $\mathcal{D}_{S_{2}}\setminus\mathcal{D}_{S_{1}}$, as $A$, $f$, $B$, $A_{1}$, $A_{2}$, in Lemma \ref{notincontradiction}, we have
\begin{equation}\label{22321eq2}
\{2,2\cdot3\}\nsubseteq\mathcal{D}_{S_{1}}\setminus\mathcal{D}_{S_{2}}\quad\text{and}\quad
\{2,2\cdot3\}\nsubseteq\mathcal{D}_{S_{2}}\setminus\mathcal{D}_{S_{1}}.
\end{equation}
Taking $\{2,2\cdot3,2^{2}\}$, $\phi(n/\cdot)$, $\{2\cdot3,2^{2}\}$, $\mathcal{D}_{S_{1}}\setminus\mathcal{D}_{S_{2}}$, $\mathcal{D}_{S_{2}}\setminus\mathcal{D}_{S_{1}}$, as $A$, $f$, $B$, $A_{1}$, $A_{2}$, in Lemma \ref{notincontradiction}, we have
\begin{equation}\label{22321eq3}
\{2\cdot3,2^{2}\}\nsubseteq\mathcal{D}_{S_{1}}\setminus\mathcal{D}_{S_{2}}\quad\text{and}\quad
\{2\cdot3,2^{2}\}\nsubseteq\mathcal{D}_{S_{2}}\setminus\mathcal{D}_{S_{1}}.
\end{equation}
Taking $\{2,2\cdot3,2^{2}\}$, $\phi(n/\cdot)$, $\{2,2^{2}\}$, $\mathcal{D}_{S_{1}}\setminus\mathcal{D}_{S_{2}}$, $\mathcal{D}_{S_{2}}\setminus\mathcal{D}_{S_{1}}$, as $A$, $f$, $B$, $A_{1}$, $A_{2}$, in Lemma \ref{notincontradiction}, we have
\begin{equation}\label{22321eq4}
\{2,2^{2}\}\nsubseteq\mathcal{D}_{S_{1}}\setminus\mathcal{D}_{S_{2}}\quad\text{and}\quad \{2,2^{2}\}\nsubseteq\mathcal{D}_{S_{2}}\setminus\mathcal{D}_{S_{1}}.
\end{equation}
(\ref{22321eq2}), (\ref{22321eq3}) and (\ref{22321eq4}) imply $\mathcal{D}_{S_{1}}\triangle\mathcal{D}_{S_{2}}\neq\{2,2\cdot3,2^{2}\}$, which is a contradiction.
\end{itemize}
\begin{table}[!h]
    \centering
{\tiny
    \caption{$n=2^{2}3^{2}$, $\mathcal{D}_{S_{1}}\setminus\mathcal{D}_{S_{2}}=\{2\}$, $\mathcal{D}_{S_{2}}\setminus\mathcal{D}_{S_{1}}=\{2^{2}\}$ and $2^{2}3\in\mathcal{D}_{S_{1}}\cap\mathcal{D}_{S_{2}}$\label{table2232}}
    \begin{tabular}{|c|c|c|}
    \hline
    $\mathcal{D}_{S_{1}}\cap\mathcal{D}_{S_{2}}$ &   $\spec(\ICG(n,\mathcal{D}_{S_{1}}))$    &
    $\spec(\ICG(n,\mathcal{D}_{S_{2}}))$  \\\hline

$\{2^{2}3\}$     &

$\left(
\begin{array}{cccc}
8 & 5 & -1 & -4\\
2 & 4 & 28 & 2
\end{array}
\right)$

    &

$\left(
\begin{array}{cc}
8 & -1\\
4 & 32
\end{array}
\right)$\\\hline

$\{2\cdot3^{2},2^{2}3\}$     &

$\left(
\begin{array}{ccccc}
9 & 4 & 0 & -2 & -5\\
2 & 4 & 16 & 12 & 2
\end{array}
\right)$

    &

$\left(
\begin{array}{cccc}
9 & 7 & 0 & -2\\
2 & 2 & 16 & 16
\end{array}
\right)$\\\hline

$\{2\cdot3,2^{2}3\}$     &

$\left(
\begin{array}{cccccc}
10 & 3 & 1 & 0 & -2 & -6\\
2 & 4 & 4 & 12 & 12 & 2
\end{array}
\right)$

    &

$\left(
\begin{array}{cccccc}
10 & 6 & 1 & 0 & -2 & -3\\
2 & 2 & 4 & 12 & 12 & 4
\end{array}
\right)$\\\hline

$\{2\cdot3,2\cdot3^{2},2^{2}3\}$     &

$\left(
\begin{array}{cccc}
11 & 2 & -1 & -7\\
2 & 8 & 24 & 2
\end{array}
\right)$

    &

$\left(
\begin{array}{ccccc}
11 & 5 & 2 & -1 & -4\\
2 & 2 & 4 & 24 & 4
\end{array}
\right)$\\\hline

$\{3^{2},2^{2}3\}$     &

$\left(
\begin{array}{ccccccc}
10 & 6 & 5 & 1 & -1 & -3 & -4\\
1 & 1 & 4 & 8 & 12 & 8 & 2
\end{array}
\right)$

    &

$\left(
\begin{array}{cccccc}
10 & 8 & 6 & 1 & -1 & -3\\
1 & 2 & 1 & 8 & 16 & 8
\end{array}
\right)$\\\hline

$\{3^{2},2\cdot3^{2},2^{2}3\}$     &

$\left(
\begin{array}{cccccc}
11 & 7 & 4 & 2 & -2 & -5\\
1 & 1 & 4 & 8 & 20 & 2
\end{array}
\right)$

    &

$\left(
\begin{array}{cccc}
11 & 7 & 2& -2\\
1 & 3 & 8 & 24
\end{array}
\right)$\\\hline

$\{3^{2},2\cdot3,2^{2}3\}$     &

$\left(
\begin{array}{ccccccc}
12 & 8 & 3 & 0 & -1 & -4 & -6\\
1 & 1 & 6 & 18 & 2 & 6 & 2
\end{array}
\right)$

    &

$\left(
\begin{array}{cccccccc}
12 & 8 & 6 & 3 & 0 & -1 & -3 & -4\\
1 & 1 & 2 & 2 & 18 & 2 & 4 & 6
\end{array}
\right)$\\\hline

$\{3^{2},2\cdot3,2\cdot3^{2},2^{2}3\}$     &

$\left(
\begin{array}{ccccccccc}
13 & 9 & 4 & 2 & 1 & 0 & -1 & -3 & -7\\
1 & 1 & 2 & 4 & 6 & 2 & 12 & 6 & 2
\end{array}
\right)$

    &

$\left(
\begin{array}{ccccccccc}
13 & 9 & 5 & 4 & 1 & 0 & -1 & -3 & -4\\
1 & 1 & 2 & 2 & 6 & 2 & 12 & 6 & 4
\end{array}
\right)$\\\hline

$\{3,2^{2}3\}$     &

$\left(
\begin{array}{ccccccccc}
12 & 5 & 4 & 3 & 1 & -1 & -3 & -4 & -5\\
1 & 4 & 1 & 2 & 6 & 12 & 6 & 2 & 2
\end{array}
\right)$

    &

$\left(
\begin{array}{cccccccc}
12 & 8 & 4 & 3 & 1 & -1 & -3 & -5\\
1& 2 & 1 & 2 & 6 & 16 & 6 & 2
\end{array}
\right)$\\\hline

$\{3,2\cdot3^{2},2^{2}3\}$     &

$\left(
\begin{array}{ccccccc}
13 & 5 & 4 & 2 & -2 & -4 & -5\\
1 & 1 & 6 & 6 & 18 & 2 & 2
\end{array}
\right)$

    &

$\left(
\begin{array}{ccccccc}
13 & 7 & 5 & 4 & 2 & -2 & -4\\
1 & 2 & 1 & 2 & 6 & 22 & 2
\end{array}
\right)$\\\hline

$\{3,2\cdot3,2^{2}3\}$     &

$\left(
\begin{array}{cccccccc}
14 & 6 & 5 & 3 & 0 & -3 & -4 & -6\\
1 & 1 & 2 & 4 & 18 & 2 & 6 & 2
\end{array}
\right)$

    &

$\left(
\begin{array}{cccccc}
14 & 6 & 5 & 0 & -3 & -4\\
1 & 3 & 2 & 18 & 6 & 6
\end{array}
\right)$\\\hline

$\{3,2\cdot3,2\cdot3^{2},2^{2}3\}$     &

$\left(
\begin{array}{ccccccccc}
15 & 7 & 6 & 2 & 1 & -1 & -2 & -3 & -7\\
1 & 1 & 2 & 4 & 6 & 12 & 2 & 6 & 2
\end{array}
\right)$

    &

$\left(
\begin{array}{ccccccccc}
15 & 7 & 6 & 5 & 1 & -1 & -2 & -3 & -4\\
1 & 1 & 2 & 2 & 6 & 12 & 2 & 6 & 4
\end{array}
\right)$\\\hline

$\{3,3^{2},2^{2}3\}$     &

$\left(
\begin{array}{cccccc}
14 & 5 & 2 & -1 & -4 & -7\\
1 & 6 & 1 & 24 & 2 & 2
\end{array}
\right)$

    &

$\left(
\begin{array}{cccccc}
14 & 8 & 5 & 2 & -1 & -7\\
1 & 2 & 2 & 1 & 28 & 2
\end{array}
\right)$\\\hline

$\{3,3^{2},2\cdot3^{2},2^{2}3\}$     &

$\left(
\begin{array}{cccccccc}
15 & 6 & 4 & 3 & 0 & -2 & -5 & -6\\
1 & 2 & 4 & 1 & 12 & 12 & 2 & 2
\end{array}
\right)$

    &

$\left(
\begin{array}{ccccccc}
15 & 7 & 6 & 3 & 0 & -2 & -6\\
1 & 2 & 2 & 1 & 12 & 16 & 2
\end{array}
\right)$\\\hline

$\{3,3^{2},2\cdot3,2^{2}3\}$     &

$\left(
\begin{array}{cccccccc}
16 & 7 & 4 & 3 & 0 & -2 & -5 & -6\\
1 & 2 & 1 & 4 & 12 & 12 & 2 & 2
\end{array}
\right)$

    &

$\left(
\begin{array}{cccccccc}
16 & 7 & 6 & 4 & 0 & -2 & -3 & -5\\
1 & 2 & 2 & 1 & 12 & 12 & 4 & 2
\end{array}
\right)$\\\hline

$\{3,3^{2},2\cdot3,2\cdot3^{2},2^{2}3\}$     &

$\left(
\begin{array}{ccccccc}
17 & 8 & 5 & 2 & -1 & -4 & -7\\
1 & 2 & 1 & 4 & 24 & 2 & 2
\end{array}
\right)$

    &

$\left(
\begin{array}{ccccc}
17 & 8 & 5 & -1 & -4\\
1 & 2 & 3& 24 & 6
\end{array}
\right)$\\\hline

$\{1,2^{2}3\}$     &

$\left(
\begin{array}{ccccc}
20 & 5 & -1 & -4 & -7\\
1 & 6 & 24 & 3 & 2
\end{array}
\right)$

    &

$\left(
\begin{array}{cccccc}
20 & 8 & 5 & -1 & -4 & -7\\
1 & 2 & 2 & 28 & 1 & 2
\end{array}
\right)$\\\hline

$\{1,2\cdot3^{2},2^{2}3\}$     &

$\left(
\begin{array}{cccccccc}
21& 6 & 4 & 0 & -2 & -3 & -5 & -6\\
1 & 2 & 4 & 12 & 12 & 1 & 2 & 2
\end{array}
\right)$

    &

$\left(
\begin{array}{ccccccc}
21 & 7 & 6 & 0 & -2 & -3 & -6\\
1 & 2 & 2 & 12 & 16 & 1 & 2
\end{array}
\right)$\\\hline

$\{1,2\cdot3,2^{2}3\}$     &

$\left(
\begin{array}{ccccccc}
22 & 7 & 3 & 0 & -2 & -5 & -6\\
1 & 2 & 4 & 12 & 13 & 2 & 2
\end{array}
\right)$

    &

$\left(
\begin{array}{ccccccc}
22 & 7  & 6 & 0 & -2 & -3 & -5\\
1 & 2 & 2 & 12 & 13 & 4 & 2
\end{array}
\right)$\\\hline

$\{1,2\cdot3,2\cdot3^{2},2^{2}3\}$     &

$\left(
\begin{array}{cccccc}
23& 8 & 2 & -1 & -4 & -7\\
1 & 2 & 4 & 25 & 2 & 2
\end{array}
\right)$

    &

$\left(
\begin{array}{ccccc}
23& 8 & 5 & -1 & -4\\
1 & 2& 2& 25 & 6
\end{array}
\right)$\\\hline

$\{1,3^{2},2^{2}3\}$     &

$\left(
\begin{array}{ccccccccc}
22 & 5 & 3 & 1 & -1 & -3 & -4 & -5 & -6\\
1 & 4 & 2 & 6 & 12 & 6 & 2 & 2 & 1
\end{array}
\right)$

    &

$\left(
\begin{array}{cccccccc}
22 & 8 & 3 & 1 & -1 & -3 & -5 & -6\\
1 & 2& 2 & 6 & 16 & 6 & 2 & 1
\end{array}
\right)$\\\hline

$\{1,3^{2},2\cdot3^{2},2^{2}3\}$     &

$\left(
\begin{array}{cccccc}
23 & 4 & 2 & -2 & -4 & -5\\
1 & 6 & 6 & 18 & 2 & 3
\end{array}
\right)$

    &

$\left(
\begin{array}{ccccccc}
23& 7 & 4 & 2 & -2 & -4 & -5\\
1 & 2 & 2 & 6 & 22 & 2 & 1
\end{array}
\right)$\\\hline

$\{1,3^{2},2\cdot3,2^{2}3\}$     &

$\left(
\begin{array}{ccccccc}
24 & 5 & 3 & 0 & -3 & -4 & -6\\
1 & 2 & 4 & 18 & 2 & 7 & 2
\end{array}
\right)$

    &

$\left(
\begin{array}{cccccc}
24 & 6 & 5 & 0 & -3 & -4\\
1 & 2& 2 & 18 & 6 & 7
\end{array}
\right)$\\\hline

$\{1,3^{2},2\cdot3,2\cdot3^{2},2^{2}3\}$     &

$\left(
\begin{array}{cccccccc}
25 & 6 & 2 & 1 & -1 & -2 & -3 & -7\\
1 & 2 & 4 & 6 & 12 & 2 & 7 &2
\end{array}
\right)$

    &

$\left(
\begin{array}{cccccccc}
25 & 6 & 5 & 1 & -1 & -2 & -3 & -4\\
1 & 2& 2 & 6 & 12 & 2 & 7 &4
\end{array}
\right)$\\\hline

$\{1,3,2^{2}3\}$     &

$\left(
\begin{array}{cccccccc}
24 & 5 & 1 & -1 & -3 & -4 & -8\\
1 & 4 & 8 & 12 & 8 & 2 & 1
\end{array}
\right)$

    &

$\left(
\begin{array}{cccccc}
24 & 8 & 1 & -1 & -3 & -8\\
1 & 2 & 8 & 16 & 8 & 1
\end{array}
\right)$\\\hline

$\{1,3,2\cdot3^{2},2^{2}3\}$     &

$\left(
\begin{array}{cccccc}
25 & 4 & 2 & -2 & -5 & -7\\
1 & 4 & 8 & 20 & 2 & 1
\end{array}
\right)$

    &

$\left(
\begin{array}{ccccc}
25 & 7 & 2 & -2 & -7\\
1 & 2 & 8 & 24 & 1
\end{array}
\right)$\\\hline

$\{1,3,2\cdot3,2^{2}3\}$     &

$\left(
\begin{array}{cccccc}
26 & 3 & 0 & -1 & -4 & -6\\
1 & 6 & 18 & 2 & 6 & 3
\end{array}
\right)$

    &

$\left(
\begin{array}{cccccccc}
26 & 6 & 3 & 0 & -1 & -3 & -4 & -6\\
1 & 2 & 2 & 18 & 2 & 4 & 6 & 1
\end{array}
\right)$\\\hline

$\{1,3,2\cdot3,2\cdot3^{2},2^{2}3\}$     &

$\left(
\begin{array}{ccccccccc}
27 & 4 & 2 & 1 & 0 & -1 & -3 & -5 & -7\\
1 & 2& 4 & 6 & 2 & 12 & 6 & 1 & 2
\end{array}
\right)$

    &

$\left(
\begin{array}{ccccccccc}
27 & 5 & 4 & 1 & 0 & -1 & -3 & -4 & -5\\
1 & 2 & 2 & 6 & 2 & 12 & 6 & 4 & 1
\end{array}
\right)$\\\hline

$\{1,3,3^{2},2^{2}3\}$     &

$\left(
\begin{array}{ccccc}
26 & 5 & -1 & -4 & -10\\
1 & 4 & 28 & 2 & 1
\end{array}
\right)$

    &

$\left(
\begin{array}{cccc}
26 & 8 & -1 & -10\\
1 & 2 & 32 & 1
\end{array}
\right)$\\\hline

$\{1,3,3^{2},2\cdot3^{2},2^{2}3\}$     &

$\left(
\begin{array}{cccccc}
27 & 4 & 0 & -2 & -5 & -9\\
1 & 4 &  16 &12 & 2 & 1
\end{array}
\right)$

    &

$\left(
\begin{array}{ccccc}
27 & 7 & 0 & -2 & -9\\
1 & 2 & 16 & 16 & 1
\end{array}
\right)$\\\hline

$\{1,3,3^{2},2\cdot3,2^{2}3\}$     &

$\left(
\begin{array}{ccccccc}
28 & 3 & 1 & 0 & -2 & -6 & -8\\
1 & 4 & 4 & 12 & 12 & 2 & 1
\end{array}
\right)$

    &

$\left(
\begin{array}{ccccccc}
28 & 6 & 1 & 0 & -2 & -3 & -8\\
1 & 2 & 4 & 12 & 12 & 4 & 1
\end{array}
\right)$\\\hline

$\{1,3,3^{2},2\cdot3,2\cdot3^{2},2^{2}3\}$     &

$\left(
\begin{array}{cccc}
29 & 2 & -1 & -7\\
1 & 8 & 24 & 3
\end{array}
\right)$

    &

$\left(
\begin{array}{cccccc}
29 & 5 & 2 & -1 & -4 & -7\\
1 & 2 & 4 & 24 & 4 & 1
\end{array}
\right)$\\\hline

    \end{tabular}

 }
 \end{table}
\begin{itemize}
\item Case 2: $|\mathcal{D}_{S_{1}}\triangle\mathcal{D}_{S_{2}}|=2$\\
Then $\mathcal{D}_{S_{1}}\triangle\mathcal{D}_{S_{2}}=\{2,2\cdot3\}$, $\{2,2^{2}\}$ or $\{2\cdot3,2^{2}\}$. Set $\mathcal{D}_{S_{1}}\triangle\mathcal{D}_{S_{2}}=\{a_{1},a_{2}\}$. Taking $\mathcal{D}_{S_{1}}\triangle\mathcal{D}_{S_{2}}$, $\phi(n/\cdot)$, $\mathcal{D}_{S_{1}}\setminus\mathcal{D}_{S_{2}}$, $\mathcal{D}_{S_{2}}\setminus\mathcal{D}_{S_{1}}$, as $A$, $f$, $A_{1}$, $A_{2}$, in Corollary \ref{notincontradictioncor}, we have either $\phi(n/a_{1})=\phi(n/a_{2})$ while $|\mathcal{D}_{S_{1}}\setminus\mathcal{D}_{S_{2}}|=|\mathcal{D}_{S_{2}}\setminus\mathcal{D}_{S_{1}}|=1$, or $\mathcal{D}_{S_{1}}\triangle\mathcal{D}_{S_{2}}=\emptyset$ leading to a contradiction. Note that $\phi(n/2)\neq\phi(n/(2\cdot3))$, that $\phi(n/2)=\phi(n/2^{2})$, and that $\phi(n/(2\cdot3))\neq\phi(n/2^{2})$. Without loss of generality, we have $\mathcal{D}_{S_{1}}\setminus\mathcal{D}_{S_{2}}=\{2\}$ and $\mathcal{D}_{S_{2}}\setminus\mathcal{D}_{S_{1}}=\{2^{2}\}$. Recall $2^{2}3\in\mathcal{D}_{S_{1}}\cap\mathcal{D}_{S_{2}}$. By Table \ref{table2232}, $\spec(\ICG(n,\mathcal{D}_{S_{1}}))\neq\spec(\ICG(n,\mathcal{D}_{S_{2}}))$, which is a contradiction.


\item Case 3: $|\mathcal{D}_{S_{1}}\triangle\mathcal{D}_{S_{2}}|=1$\\
Recall (\ref{22321eq1}). Taking $\mathcal{D}_{S_{1}}\triangle\mathcal{D}_{S_{2}}$, $\phi(n/\cdot)$, $\mathcal{D}_{S_{1}}\setminus\mathcal{D}_{S_{2}}$, $\mathcal{D}_{S_{2}}\setminus\mathcal{D}_{S_{1}}$, as $A$, $f$, $A_{1}$, $A_{2}$, in Lemma \ref{notincontradictionnew}, we have $\mathcal{D}_{S_{1}}\triangle\mathcal{D}_{S_{2}}=\emptyset$, which is a contradiction.
\end{itemize}
This completes the proof.\qed
\end{proof}

\begin{theorem}\label{2232}
Set $n=2^{2}3^{2}$. Let $\mathcal{D}_{S_{1}}$ and $\mathcal{D}_{S_{2}}$ be two subsets of $\mathcal{D}_{[n]}\setminus\{n\}$. Then $\spec(\ICG(n,\mathcal{D}_{S_{1}}))
=\spec(\ICG(n,\mathcal{D}_{S_{2}}))$ implies $\mathcal{D}_{S_{1}}=\mathcal{D}_{S_{2}}$.

\end{theorem}
\begin{proof}
By Lemma \ref{n/pnotintriangle}, $2^{2}3\notin\mathcal{D}_{S_{1}}\triangle\mathcal{D}_{S_{2}}$. By Lemma \ref{22321} and Corollary \ref{regularisocor}, we obtain the result.\qed

\end{proof}

\begin{theorem}\label{22q2}
Set $n=2^{2}q^{2}$ with prime $q>3$. Let $\mathcal{D}_{S_{1}}$ and $\mathcal{D}_{S_{2}}$ be two subsets of $\mathcal{D}_{[n]}\setminus\{n\}$. Then $\spec(\ICG(n,\mathcal{D}_{S_{1}}))
=\spec(\ICG(n,\mathcal{D}_{S_{2}}))$ implies $\mathcal{D}_{S_{1}}=\mathcal{D}_{S_{2}}$.
\end{theorem}

\begin{proof}
In order to prove $\mathcal{D}_{S_{1}}=\mathcal{D}_{S_{2}}$, similar to the proof of Lemma \ref{22321}, replacing $3$ by $q$, we only need to rule out the case where $\mathcal{D}_{S_{1}}\setminus\mathcal{D}_{S_{2}}=\{2\}$ and $\mathcal{D}_{S_{2}}\setminus\mathcal{D}_{S_{1}}=\{2^{2}\}$. Set
\begin{align*}
\spec(\ICG(n,\mathcal{D}_{S_{1}}))
=\spec(\ICG(n,\mathcal{D}_{S_{2}}))
=\left(
\begin{array}{cccc}
\nu_{1} & \nu_{2} & \ldots & \nu_{J} \\
m_{1} & m_{2} & \ldots & m_{J}
\end{array}
\right).
\end{align*}
By simple calculation, $\forall d\in\{1,2,2^{2},q,2q,2^{2}q\}$, we have $(q-1)|\phi(n/d)$. The following part is similar to the proof of Lemma \ref{lambdan/2}. Set $\nu_{j_{0}}=\lambda_{q^{2}}(S_{1})$. Then
\begin{align*}
|\mathcal{L}_{S_{1}}(\nu_{j_{0}})\setminus\{2q^{2},2^{2}q^{2}\}|
&=\sum_{d\in\mathcal{D}_{\mathcal{L}_{S_{1}}(\nu_{j_{0}})}\setminus\{2q^{2},2^{2}q^{2}\}}\phi(n/d)\tag{by (\ref{Scard})}
\\&=\phi(n/q^{2})
+\sum_{d\in\mathcal{D}_{\mathcal{L}_{S_{1}}(\nu_{j_{0}})}\setminus\{q^{2},2q^{2},2^{2}q^{2}\}}\phi(n/d)
\\&\equiv2~(\mathrm{mod}\; (q-1))\tag{because $\forall d\in\{1,2,2^{2},q,2q,2^{2}q\}$, $(q-1)|\phi(n/d)$}
\end{align*}
and $\forall j\in[J]\setminus\{j_{0}\}$,
\begin{align*}
|\mathcal{L}_{S_{1}}(\nu_{j})\setminus\{2q^{2},2^{2}q^{2}\}|
&=\sum_{d\in\mathcal{D}_{\mathcal{L}_{S_{1}}(\nu_{j})}\setminus\{2q^{2},2^{2}q^{2}\}}\phi(n/d)\tag{by (\ref{Scard})}
\\&=\sum_{d\in\mathcal{D}_{\mathcal{L}_{S_{1}}(\nu_{j})}\setminus\{q^{2},2q^{2},2^{2}q^{2}\}}\phi(n/d)\tag{because $\lambda_{q^{2}}(S_{1})=\nu_{j_{0}}\neq\nu_{j}$}
\\&\equiv0~(\mathrm{mod}\; (q-1)).\tag{because $\forall d\in\{1,2,2^{2},q,2q,2^{2}q\}$, $(q-1)|\phi(n/d)$}
\end{align*}
Set $\nu_{j_{0}'}=\lambda_{q^{2}}(S_{2})$. Similarly,
$$|\mathcal{L}_{S_{2}}(\nu_{j_{0}'})\setminus\{2q^{2},2^{2}q^{2}\}|\equiv2~(\mathrm{mod}\; (q-1)).$$
By Lemmas \ref{lambdan} and \ref{lambdan/2}, we have $\lambda_{n}(S_{1})=\lambda_{n}(S_{2})$ and $\lambda_{\frac{n}{2}}(S_{1})=\lambda_{\frac{n}{2}}(S_{2})$. By Lemma \ref{minusR}, $\forall j\in[J]$,
$$
|\mathcal{L}_{S_{1}}(\nu_{j})\setminus\{2q^{2},2^{2}q^{2}\}|
=|\mathcal{L}_{S_{2}}(\nu_{j})\setminus\{2q^{2},2^{2}q^{2}\}|.
$$
In particular,
$$
|\mathcal{L}_{S_{1}}(\nu_{j_{0}'})\setminus\{2q^{2},2^{2}q^{2}\}|
=|\mathcal{L}_{S_{2}}(\nu_{j_{0}'})\setminus\{2q^{2},2^{2}q^{2}\}|
\equiv2~(\mathrm{mod}\; (q-1)).
$$
Therefore, $j_{0}=j_{0}'$ and so
$$\lambda_{q^{2}}(S_{1})=\nu_{j_{0}}=\nu_{j_{0}'}=\lambda_{q^{2}}(S_{2}).$$
Note that $\mathcal{D}_{S_{1}}=(\mathcal{D}_{S_{1}}\setminus\mathcal{D}_{S_{2}})\cup(\mathcal{D}_{S_{1}}\cap\mathcal{D}_{S_{2}})$ and that $\mathcal{D}_{S_{2}}=(\mathcal{D}_{S_{2}}\setminus\mathcal{D}_{S_{1}})\cup(\mathcal{D}_{S_{1}}\cap\mathcal{D}_{S_{2}})$. We have
\begin{align*}
\lambda_{q^{2}}(S_{1})
&=\mathcal{R}_{n/2}(q^{2})
+\sum_{d\in\mathcal{D}_{S_{1}}\cap\mathcal{D}_{S_{2}}}\mathcal{R}_{n/d}(q^{2})\tag{by (\ref{ramanujan})}
\\&=-q(q-1)+\sum_{d\in\mathcal{D}_{S_{1}}\cap\mathcal{D}_{S_{2}}}\mathcal{R}_{n/d}(q^{2})\tag{by (\ref{lambdaphimu})}
\\&<q(q-1)+\sum_{d\in\mathcal{D}_{S_{1}}\cap\mathcal{D}_{S_{2}}}\mathcal{R}_{n/d}(q^{2})
\\&=\mathcal{R}_{n/2^{2}}(q^{2})
+\sum_{d\in\mathcal{D}_{S_{1}}\cap\mathcal{D}_{S_{2}}}\mathcal{R}_{n/d}(q^{2})\tag{by (\ref{lambdaphimu})}
\\&=\lambda_{q^{2}}(S_{2})\tag{by (\ref{ramanujan})},
\end{align*}
which is a contradiction.\qed

\end{proof}

\begin{lemma}\label{3272}
Set $n=3^{2}7^{2}$. Let $\mathcal{D}_{S_{1}}$ and $\mathcal{D}_{S_{2}}$ be two subsets of $\mathcal{D}_{[n]}\setminus\{n\}$ such that
\begin{itemize}
\item[\rm{(1)}]$\mathcal{D}_{S_{1}}\setminus\mathcal{D}_{S_{2}}=\{3\}$ and $\mathcal{D}_{S_{2}}\setminus\mathcal{D}_{S_{1}}=\{3^{2},7,7^{2}\}$; and
\item[\rm{(2)}]$1\in\mathcal{D}_{S_{1}}\cap\mathcal{D}_{S_{2}}$.
\end{itemize}
Then $\spec(\ICG(n,\mathcal{D}_{S_{1}}))
\neq\spec(\ICG(n,\mathcal{D}_{S_{2}}))$.
\end{lemma}
\begin{proof}
The 8 pairs of spectra listed in Table \ref{table3272} suggest the result.
    \begin{table}[!h]
    \centering
{\tiny
    \caption{$n=3^{2}7^{2}$, $\mathcal{D}_{S_{1}}\setminus\mathcal{D}_{S_{2}}=\{3\}$, $\mathcal{D}_{S_{2}}\setminus\mathcal{D}_{S_{1}}=\{3^{2},7,7^{2}\}$ and $1\in\mathcal{D}_{S_{1}}\cap\mathcal{D}_{S_{2}}$\label{table3272}}
    \begin{tabular}{|c|c|c|}

    \hline
    $\mathcal{D}_{S_{1}}\cap\mathcal{D}_{S_{2}}$ &   $\spec(\ICG(n,\mathcal{D}_{S_{1}}))$    &
    $\spec(\ICG(n,\mathcal{D}_{S_{2}}))$  \\\hline
$\{1\}$     &

$\left(
\begin{array}{ccccc}
336 & 7 & 0 & -42 & -56\\
1 & 48 & 378 & 8 & 6
\end{array}
\right)$

    &

$\left(
\begin{array}{ccccc}
336 & 42 & 0 & -7 & -105\\
1 & 6 & 378 & 54 & 2
\end{array}
\right)$\\\hline

$\{1,3^{2}7\}$     &

$\left(
\begin{array}{ccccc}
342 & 13 & -1 & -36 & -50\\
1 & 48 & 378 & 8 & 6
\end{array}
\right)$

    &

$\left(
\begin{array}{cccc}
342& 48 & -1 & -99\\
1 & 6 & 432 & 2
\end{array}
\right)$\\\hline

$\{1,3\cdot7^{2}\}$     &

$\left(
\begin{array}{cccccccc}
338 & 9 & 6 & 2 & -1 & -40 & -43 & -54\\
1 & 12 & 36 & 126 & 252 & 2 & 6 & 6
\end{array}
\right)$

    &

$\left(
\begin{array}{ccccccc}
338 & 41 & 2 & -1 & -5 & -8 & -103\\
1 & 6 & 126 & 252 & 18 & 36 & 2
\end{array}
\right)$\\\hline

$\{1,3\cdot7^{2},3^{2}7\}$     &

$\left(
\begin{array}{cccccccc}
344 & 15 &12 & 1 & -2 & -34 & -37 & -48\\
1 & 12 & 36 & 126 & 252 & 2 & 6 & 6
\end{array}
\right)$

    &

$\left(
\begin{array}{ccccc}
344 & 47 & 1 & -2 & -97\\
1 & 6 & 144 & 288 & 2
\end{array}
\right)$\\\hline

$\{1,3\cdot7\}$     &

$\left(
\begin{array}{ccccccc}
348 & 19 & 1 & -2 & -30 & -44 & -48\\
1 & 12 & 288 & 126 & 2  & 6 & 6
\end{array}
\right)$

    &

$\left(
\begin{array}{ccccccc}
348 & 36 & 5 & 1 & -2 & -13 & -93\\
1 & 6 & 18 & 252 & 126 & 36 & 2
\end{array}
\right)$\\\hline

$\{1,3\cdot7,3^{2}7\}$     &

$\left(
\begin{array}{cccccccc}
354 & 25 & 7 & 0 & -3 & -24 & -38 & -42\\
1 & 12 & 36 & 252 & 126 & 2 & 6 & 6
\end{array}
\right)$

    &

$\left(
\begin{array}{ccccccc}
354& 42 & 11 & 0 & -3 & -7 & -87\\
1 & 6 & 18 & 252 & 126 & 36 & 2
\end{array}
\right)$\\\hline

$\{1,3\cdot7,3\cdot7^{2}\}$     &

$\left(
\begin{array}{cccccc}
350 & 21 & 0 & -28 & -42 & -49\\
1 & 12 & 414& 2 & 6 & 6
\end{array}
\right)$

    &

$\left(
\begin{array}{cccccc}
350 & 35 & 7 & 0 & -14 & -91\\
1 & 6 & 18 & 378 & 36 & 2
\end{array}
\right)$\\\hline

$\{1,3\cdot7,3\cdot7^{2},3^{2}7\}$     &

$\left(
\begin{array}{ccccccc}
356& 27 & 6 & -1 & -22 & -36 & -43\\
1 & 12 & 36 & 378 & 2 & 6 & 6
\end{array}
\right)$

    &

$\left(
\begin{array}{cccccc}
356 & 41 & 13 & -1 & -8 & -85\\
1 & 6 & 18 & 378 & 36 & 2
\end{array}
\right)$\\\hline

    \end{tabular}

 }
 \end{table}
This completes the proof.\qed

\end{proof}


\begin{lemma}\label{p2q21}
Set $n=p^{2}q^{2}$ with primes $3\leq p<q$. Let $\mathcal{D}_{S_{1}}$ and $\mathcal{D}_{S_{2}}$ be two subsets of $\mathcal{D}_{[n]}\setminus\{n\}$ such that $1\in\mathcal{D}_{S_{1}}\cap\mathcal{D}_{S_{2}}$. Then $\spec(\ICG(n,\mathcal{D}_{S_{1}}))
=\spec(\ICG(n,\mathcal{D}_{S_{2}}))$ implies $\mathcal{D}_{S_{1}}=\mathcal{D}_{S_{2}}$.
\end{lemma}
\begin{proof}
By Lemma \ref{lambdan}, $\lambda_{n}(S_{1})=\lambda_{n}(S_{2})$. By (\ref{ramanujan}) and (\ref{lambdaphimu}), $\sum_{d\in\mathcal{D}_{S_{1}}}\phi(n/d)=\sum_{d\in\mathcal{D}_{S_{2}}}\phi(n/d)$ and so
\begin{equation}\label{p2q21eq1}
\sum_{d\in\mathcal{D}_{S_{1}}\setminus\mathcal{D}_{S_{2}}}\phi(n/d)=\sum_{d\in\mathcal{D}_{S_{2}}\setminus\mathcal{D}_{S_{1}}}\phi(n/d).
\end{equation}
By Lemma \ref{n/pnotintriangle}, $p^{2}q,pq^{2}\notin\mathcal{D}_{S_{1}}\triangle\mathcal{D}_{S_{2}}$ and so $\mathcal{D}_{S_{1}}\triangle\mathcal{D}_{S_{2}}\subseteq\{q,q^{2},p,pq,p^{2}\}$. Since $\phi(n)/n=\frac{(p-1)(q-1)}{pq}\geq\frac{2\cdot4}{3\cdot5}>\frac{1}{2}$, we have $\phi(n)>n/2$. By Corollary \ref{1equal}, $\lambda_{1}(S_{1})=\lambda_{1}(S_{2})$. Hence, by (\ref{ramanujan}) and (\ref{lambdaphimu}), $\sum_{d\in\mathcal{D}_{S_{1}}}\mu(n/d)=\sum_{d\in\mathcal{D}_{S_{2}}}\mu(n/d)$ and so
$$
\sum_{d\in\mathcal{D}_{S_{1}}\setminus\mathcal{D}_{S_{2}}}\mu(n/d)=\sum_{d\in\mathcal{D}_{S_{2}}\setminus\mathcal{D}_{S_{1}}}\mu(n/d).
$$
Note that $\mu(n/pq)=1$ and that $\forall d\in\{q,q^{2},p,p^{2}\}$, $\mu(n/d)=0$. Taking $\{q,q^{2},p,pq,p^{2}\}$, $\mu(n/\cdot)$, $\{pq\}$, $\mathcal{D}_{S_{1}}\setminus\mathcal{D}_{S_{2}}$, $\mathcal{D}_{S_{2}}\setminus\mathcal{D}_{S_{1}}$, as $A$, $f$, $B$, $A_{1}$, $A_{2}$, in Lemma \ref{notincontradiction}, we have $pq\notin\mathcal{D}_{S_{1}}\triangle\mathcal{D}_{S_{2}}$ and so
$$\mathcal{D}_{S_{1}}\triangle\mathcal{D}_{S_{2}}\subseteq\{q,q^{2},p,p^{2}\}.$$
To prove $\mathcal{D}_{S_{1}}=\mathcal{D}_{S_{2}}$, we rule out the following 4 cases.
\begin{itemize}
\item Case 1: $|\mathcal{D}_{S_{1}}\triangle\mathcal{D}_{S_{2}}|=4$\\
Then $\mathcal{D}_{S_{1}}\triangle\mathcal{D}_{S_{2}}=\{q,q^{2},p,p^{2}\}$. Note that $\phi(n/p)>\phi(n/p^{2})>\phi(n/q^{2})$ and that $\phi(n/p)>\phi(n/q)>\phi(n/q^{2})$. Recall (\ref{p2q21eq1}). Taking $\{q,q^{2},p,p^{2}\}$, $\phi(n/\cdot)$, $\{p,q\}$, $\mathcal{D}_{S_{1}}\setminus\mathcal{D}_{S_{2}}$, $\mathcal{D}_{S_{2}}\setminus\mathcal{D}_{S_{1}}$, as $A$, $f$, $B$, $A_{1}$, $A_{2}$, in Lemma \ref{notincontradiction}, we have
\begin{equation}\label{p2q21eq2}
\{p,q\}\nsubseteq\mathcal{D}_{S_{1}}\setminus\mathcal{D}_{S_{2}}\quad\text{and}\quad \{p,q\}\nsubseteq\mathcal{D}_{S_{2}}\setminus\mathcal{D}_{S_{1}}.
\end{equation}
Taking $\{q,q^{2},p,p^{2}\}$, $\phi(n/\cdot)$, $\{p,p^{2}\}$, $\mathcal{D}_{S_{1}}\setminus\mathcal{D}_{S_{2}}$, $\mathcal{D}_{S_{2}}\setminus\mathcal{D}_{S_{1}}$, as $A$, $f$, $B$, $A_{1}$, $A_{2}$, in Lemma \ref{notincontradiction}, we have
\begin{equation}\label{p2q21eq3}
\{p,p^{2}\}\nsubseteq\mathcal{D}_{S_{1}}\setminus\mathcal{D}_{S_{2}}\quad\text{and}\quad \{p,p^{2}\}\nsubseteq\mathcal{D}_{S_{2}}\setminus\mathcal{D}_{S_{1}}.
\end{equation}
Without loss of generality, we have
\begin{itemize}
\item Subcase 1.1: $\mathcal{D}_{S_{1}}\setminus\mathcal{D}_{S_{2}}=\{q,q^{2},p,p^{2}\}$ and $\mathcal{D}_{S_{2}}\setminus\mathcal{D}_{S_{1}}=\emptyset$\\
    This contradicts (\ref{p2q21eq2}) and (\ref{p2q21eq3}).
\item Subcase 1.2: $\mathcal{D}_{S_{1}}\setminus\mathcal{D}_{S_{2}}=\{q,q^{2},p\}$ and $\mathcal{D}_{S_{2}}\setminus\mathcal{D}_{S_{1}}=\{p^{2}\}$\\
    This contradicts (\ref{p2q21eq2}).
\item Subcase 1.3: $\mathcal{D}_{S_{1}}\setminus\mathcal{D}_{S_{2}}=\{q,p,p^{2}\}$ and $\mathcal{D}_{S_{2}}\setminus\mathcal{D}_{S_{1}}=\{q^{2}\}$\\
    This contradicts (\ref{p2q21eq2}) and (\ref{p2q21eq3}).
\item Subcase 1.4: $\mathcal{D}_{S_{1}}\setminus\mathcal{D}_{S_{2}}=\{q^{2},p,p^{2}\}$ and $\mathcal{D}_{S_{2}}\setminus\mathcal{D}_{S_{1}}=\{q\}$\\
    This contradicts (\ref{p2q21eq3}).

\item Subcase 1.5: $\mathcal{D}_{S_{1}}\setminus\mathcal{D}_{S_{2}}=\{q,q^{2},p^{2}\}$ and $\mathcal{D}_{S_{2}}\setminus\mathcal{D}_{S_{1}}=\{p\}$\\
    Then (\ref{p2q21eq1}) implies that $(p-2)(q-1)=p(p-1)$. Set $q=p+k$. Then $k$ is a positive even integer. If $k=2$, then, by simple calculation, $(p-2)(q-1)=p(p-1)$ leads to a contradiction. If $k\geq4$, then $(p-2)(q-1)=p(p-1)$ implies $p=\frac{2k-2}{k-2}<4$. Hence $p=3$ and $q=7$. Recall that $1\in\mathcal{D}_{S_{1}}\cap\mathcal{D}_{S_{2}}$. By Lemma \ref{3272}, $\spec(\ICG(n,\mathcal{D}_{S_{1}}))\neq\spec(\ICG(n,\mathcal{D}_{S_{2}}))$, which is a contradiction.

\item Subcase 1.6: $\mathcal{D}_{S_{1}}\setminus\mathcal{D}_{S_{2}}=\{q,p\}$ and $\mathcal{D}_{S_{2}}\setminus\mathcal{D}_{S_{1}}=\{q^{2},p^{2}\}$\\
    This contradicts (\ref{p2q21eq2}).

\item Subcase 1.7: $\mathcal{D}_{S_{1}}\setminus\mathcal{D}_{S_{2}}=\{q^{2},p\}$ and $\mathcal{D}_{S_{2}}\setminus\mathcal{D}_{S_{1}}=\{q,p^{2}\}$\\
    Then (\ref{p2q21eq1}) implies that $(p-1)q(q-1)+p(p-1)=q(q-1)+p(p-1)(q-1)$. Set $q=p+k$. Then $k$ is a positive even integer and $k=\frac{-p^{2}+4p-2}{p-2}$. If $p=3$, then $k=1$ is odd, which is a contradiction. If $p>3$, then $k<0$, which is a contradiction.
\item Subcase 1.8: $\mathcal{D}_{S_{1}}\setminus\mathcal{D}_{S_{2}}=\{p,p^{2}\}$ and $\mathcal{D}_{S_{2}}\setminus\mathcal{D}_{S_{1}}=\{q,q^{2}\}$\\
    This contradicts (\ref{p2q21eq3}).

\end{itemize}
\item Case 2: $|\mathcal{D}_{S_{1}}\triangle\mathcal{D}_{S_{2}}|=3$\\
Then we have $\mathcal{D}_{S_{1}}\triangle\mathcal{D}_{S_{2}}=\{q,q^{2},p\}$, $\{q,q^{2},p^{2}\}$, $\{q,p,p^{2}\}$, or $\{q^{2},p,p^{2}\}$.

Suppose that $\mathcal{D}_{S_{1}}\triangle\mathcal{D}_{S_{2}}=\{q,q^{2},p\}$. Note that $\phi(n/p)>\phi(n/q)>\phi(n/q^{2})$. Recall (\ref{p2q21eq1}). Taking $\{q,q^{2},p\}$, $\phi(n/\cdot)$, $\{p,q\}$, $\mathcal{D}_{S_{1}}\setminus\mathcal{D}_{S_{2}}$, $\mathcal{D}_{S_{2}}\setminus\mathcal{D}_{S_{1}}$, as $A$, $f$, $B$, $A_{1}$, $A_{2}$, in Lemma \ref{notincontradiction}, we have
\begin{equation}\label{p2q21eq4}
\{p,q\}\nsubseteq\mathcal{D}_{S_{1}}\setminus\mathcal{D}_{S_{2}}\quad\text{and}\quad \{p,q\}\nsubseteq\mathcal{D}_{S_{2}}\setminus\mathcal{D}_{S_{1}}.
\end{equation}
Taking $\{q,q^{2},p\}$, $\phi(n/\cdot)$, $\{p,q^{2}\}$, $\mathcal{D}_{S_{1}}\setminus\mathcal{D}_{S_{2}}$, $\mathcal{D}_{S_{2}}\setminus\mathcal{D}_{S_{1}}$, as $A$, $f$, $B$, $A_{1}$, $A_{2}$, in Lemma \ref{notincontradiction}, we have
\begin{equation}\label{p2q21eq5}
\{p,q^{2}\}\nsubseteq\mathcal{D}_{S_{1}}\setminus\mathcal{D}_{S_{2}}\quad\text{and}\quad \{p,q^{2}\}\nsubseteq\mathcal{D}_{S_{2}}\setminus\mathcal{D}_{S_{1}}.
\end{equation}
Without loss of generality, we have
\begin{itemize}
\item Subcase 2.1: $\mathcal{D}_{S_{1}}\setminus\mathcal{D}_{S_{2}}=\{q,q^{2},p\}$ and $\mathcal{D}_{S_{2}}\setminus\mathcal{D}_{S_{1}}=\emptyset$\\
This contradicts (\ref{p2q21eq4}) and (\ref{p2q21eq5}).

\item Subcase 2.2: $\mathcal{D}_{S_{1}}\setminus\mathcal{D}_{S_{2}}=\{q,p\}$ and $\mathcal{D}_{S_{2}}\setminus\mathcal{D}_{S_{1}}=\{q^{2}\}$\\
This contradicts (\ref{p2q21eq4}).

\item Subcase 2.3: $\mathcal{D}_{S_{1}}\setminus\mathcal{D}_{S_{2}}=\{q^{2},p\}$ and $\mathcal{D}_{S_{2}}\setminus\mathcal{D}_{S_{1}}=\{q\}$\\
This contradicts (\ref{p2q21eq5}).

\item Subcase 2.4: $\mathcal{D}_{S_{1}}\setminus\mathcal{D}_{S_{2}}=\{q,q^{2}\}$ and $\mathcal{D}_{S_{2}}\setminus\mathcal{D}_{S_{1}}=\{p\}$\\
Then (\ref{p2q21eq1}) implies that $q=p+1$. Since both $p$ and $q$ are odd, we have a contradiction.
\end{itemize}

Suppose that $\mathcal{D}_{S_{1}}\triangle\mathcal{D}_{S_{2}}=\{q,q^{2},p^{2}\}$. Note that $\phi(n/q)>\phi(n/q^{2})$ and that $\phi(n/p^{2})>\phi(n/q^{2})$. Recall (\ref{p2q21eq1}). Taking $\{q,q^{2},p^{2}\}$, $\phi(n/\cdot)$, $\{q,p^{2}\}$, $\mathcal{D}_{S_{1}}\setminus\mathcal{D}_{S_{2}}$, $\mathcal{D}_{S_{2}}\setminus\mathcal{D}_{S_{1}}$, as $A$, $f$, $B$, $A_{1}$, $A_{2}$, in Lemma \ref{notincontradiction}, we have
\begin{equation}\label{p2q21eq6}
\{q,p^{2}\}\nsubseteq\mathcal{D}_{S_{1}}\setminus\mathcal{D}_{S_{2}}\quad\text{and}\quad \{q,p^{2}\}\nsubseteq\mathcal{D}_{S_{2}}\setminus\mathcal{D}_{S_{1}}.
\end{equation}
Without loss of generality, we have
\begin{itemize}
\item Subcase 2.5: $\mathcal{D}_{S_{1}}\setminus\mathcal{D}_{S_{2}}=\{q,q^{2},p^{2}\}$ and $\mathcal{D}_{S_{2}}\setminus\mathcal{D}_{S_{1}}=\emptyset$\\
This contradicts (\ref{p2q21eq6}).
\item Subcase 2.6: $\mathcal{D}_{S_{1}}\setminus\mathcal{D}_{S_{2}}=\{q,p^{2}\}$ and $\mathcal{D}_{S_{2}}\setminus\mathcal{D}_{S_{1}}=\{q^{2}\}$\\
This contradicts (\ref{p2q21eq6}).
\item Subcase 2.7: $\mathcal{D}_{S_{1}}\setminus\mathcal{D}_{S_{2}}=\{p^{2},q^{2}\}$ and $\mathcal{D}_{S_{2}}\setminus\mathcal{D}_{S_{1}}=\{q\}$\\
Then (\ref{p2q21eq1}) implies that $p(p-1)(q-2)=q(q-1)$. Then $q|(p-1)(q-2)=(p-1)q-2(p-1)$ and so $q|2(p-1)$. Since $q>p-1$, we have $q=2(p-1)$. Since $q$ is odd, we have a contradiction.
\item Subcase 2.8: $\mathcal{D}_{S_{1}}\setminus\mathcal{D}_{S_{2}}=\{q,q^{2}\}$ and $\mathcal{D}_{S_{2}}\setminus\mathcal{D}_{S_{1}}=\{p^{2}\}$\\
Then (\ref{p2q21eq1}) implies that
\begin{equation}\label{p2q21eq7}
q=p^{2}-p+1.
\end{equation}
Besides, we have
\begin{align*}
\lambda_{p}(S_{1})
&=\mathcal{R}_{n/q}(p)+\mathcal{R}_{n/q^{2}}(p)
+\sum_{d\in\mathcal{D}_{S_{1}}\cap\mathcal{D}_{S_{2}}}\mathcal{R}_{n/d}(p)\tag{by (\ref{ramanujan})}
\\&=p-p+\sum_{d\in\mathcal{D}_{S_{1}}\cap\mathcal{D}_{S_{2}}}\mathcal{R}_{n/d}(p)\tag{by (\ref{lambdaphimu})}
\\&=0+\sum_{d\in\mathcal{D}_{S_{1}}\cap\mathcal{D}_{S_{2}}}\mathcal{R}_{n/d}(p)
\\&=\mathcal{R}_{n/p^{2}}(p)
+\sum_{d\in\mathcal{D}_{S_{1}}\cap\mathcal{D}_{S_{2}}}\mathcal{R}_{n/d}(p)\tag{by (\ref{lambdaphimu})}
\\&=\lambda_{p}(S_{2})\tag{by (\ref{ramanujan})}
\end{align*}
and
\begin{align*}
\lambda_{p^{2}}(S_{1})
&=\mathcal{R}_{n/q}(p^{2})+\mathcal{R}_{n/q^{2}}(p^{2})
+\sum_{d\in\mathcal{D}_{S_{1}}\cap\mathcal{D}_{S_{2}}}\mathcal{R}_{n/d}(p^{2})\tag{by (\ref{ramanujan})}
\\&=-p(p-1)+p(p-1)+\sum_{d\in\mathcal{D}_{S_{1}}\cap\mathcal{D}_{S_{2}}}\mathcal{R}_{n/d}(p^{2})\tag{by (\ref{lambdaphimu})}
\\&=0+\sum_{d\in\mathcal{D}_{S_{1}}\cap\mathcal{D}_{S_{2}}}\mathcal{R}_{n/d}(p^{2})
\\&=\mathcal{R}_{n/p^{2}}(p^{2})
+\sum_{d\in\mathcal{D}_{S_{1}}\cap\mathcal{D}_{S_{2}}}\mathcal{R}_{n/d}(p^{2})\tag{by (\ref{lambdaphimu})}
\\&=\lambda_{p^{2}}(S_{2}).\tag{by (\ref{ramanujan})}
\end{align*}
Since $p\geq3$, we have $p^{2}(p^{2}-3p+1)>-1$. By (\ref{p2q21eq7}), the inequality, $p^{2}(p^{2}-3p+1)>-1$, implies
$$
\phi(n/q)>n-\sum_{d\in\{n,1,p,p^{2}\}}\phi(n/d)-\phi(n/q).
$$
Taking $\{n,1,p,p^{2}\}$, $\{q\}$, as $\mathcal{D}_{R}$, $\mathcal{D}_{T}$, in Lemma \ref{mcontradiction}, we have
$$\lambda_{q}(S_{1})=\lambda_{q}(S_{2}).$$
However,
\begin{align*}
\lambda_{q}(S_{1})
&=\mathcal{R}_{n/q}(q)+\mathcal{R}_{n/q^{2}}(q)
+\sum_{d\in\mathcal{D}_{S_{1}}\cap\mathcal{D}_{S_{2}}}\mathcal{R}_{n/d}(q)\tag{by (\ref{ramanujan})}
\\&=0+0+\sum_{d\in\mathcal{D}_{S_{1}}\cap\mathcal{D}_{S_{2}}}\mathcal{R}_{n/d}(q)\tag{by (\ref{lambdaphimu})}
\\&>-q+\sum_{d\in\mathcal{D}_{S_{1}}\cap\mathcal{D}_{S_{2}}}\mathcal{R}_{n/d}(q)
\\&=\mathcal{R}_{n/p^{2}}(q)
+\sum_{d\in\mathcal{D}_{S_{1}}\cap\mathcal{D}_{S_{2}}}\mathcal{R}_{n/d}(q)\tag{by (\ref{lambdaphimu})}
\\&=\lambda_{q}(S_{2})\tag{by (\ref{ramanujan})},
\end{align*}
which is a contradiction.
\end{itemize}

Suppose that $\mathcal{D}_{S_{1}}\triangle\mathcal{D}_{S_{2}}=\{q,p,p^{2}\}$. Note that $\phi(n/p)>\phi(n/q)$ and that $\phi(n/p)>\phi(n/p^{2})$. Recall (\ref{p2q21eq1}). Taking $\{q,p,p^{2}\}$, $\phi(n/\cdot)$, $\{q,p\}$, $\mathcal{D}_{S_{1}}\setminus\mathcal{D}_{S_{2}}$, $\mathcal{D}_{S_{2}}\setminus\mathcal{D}_{S_{1}}$, as $A$, $f$, $B$, $A_{1}$, $A_{2}$, in Lemma \ref{notincontradiction}, we have
\begin{equation}\label{p2q21eq8}
\{q,p\}\nsubseteq\mathcal{D}_{S_{1}}\setminus\mathcal{D}_{S_{2}}\quad\text{and}\quad \{q,p\}\nsubseteq\mathcal{D}_{S_{2}}\setminus\mathcal{D}_{S_{1}}.
\end{equation}
Recall (\ref{p2q21eq1}). Taking $\{q,p,p^{2}\}$, $\phi(n/\cdot)$, $\{p,p^{2}\}$, $\mathcal{D}_{S_{1}}\setminus\mathcal{D}_{S_{2}}$, $\mathcal{D}_{S_{2}}\setminus\mathcal{D}_{S_{1}}$, as $A$, $f$, $B$, $A_{1}$, $A_{2}$, in Lemma \ref{notincontradiction}, we have
\begin{equation}\label{p2q21eq9}
\{p,p^{2}\}\nsubseteq\mathcal{D}_{S_{1}}\setminus\mathcal{D}_{S_{2}}\quad\text{and}\quad \{p,p^{2}\}\nsubseteq\mathcal{D}_{S_{2}}\setminus\mathcal{D}_{S_{1}}.
\end{equation}
Without loss of generality, we have
\begin{itemize}
\item Subcase 2.9: $\mathcal{D}_{S_{1}}\setminus\mathcal{D}_{S_{2}}=\{q,p,p^{2}\}$ and $\mathcal{D}_{S_{2}}\setminus\mathcal{D}_{S_{1}}=\emptyset$\\
This contradicts (\ref{p2q21eq8}) and (\ref{p2q21eq9}).

\item Subcase 2.10: $\mathcal{D}_{S_{1}}\setminus\mathcal{D}_{S_{2}}=\{q,p\}$ and $\mathcal{D}_{S_{2}}\setminus\mathcal{D}_{S_{1}}=\{p^{2}\}$\\
This contradicts (\ref{p2q21eq8}).

\item Subcase 2.11: $\mathcal{D}_{S_{1}}\setminus\mathcal{D}_{S_{2}}=\{p,p^{2}\}$ and $\mathcal{D}_{S_{2}}\setminus\mathcal{D}_{S_{1}}=\{q\}$\\
This contradicts (\ref{p2q21eq9}).

\item Subcase 2.12: $\mathcal{D}_{S_{1}}\setminus\mathcal{D}_{S_{2}}=\{p^{2},q\}$ and $\mathcal{D}_{S_{2}}\setminus\mathcal{D}_{S_{1}}=\{p\}$\\
Then (\ref{p2q21eq1}) implies that $(p-1)q(q-1)=q(q-1)+p(p-1)(q-1)$. Set $q=p+k$. Then $k$ is a positive even integer and $p=\frac{2k}{k-1}$. If $k=2$, then $p=4$, which is a contradiction. If $k\geq4$, then $p=\frac{2k}{k-1}<3$, which is a contradiction.
\end{itemize}

Now we have $\mathcal{D}_{S_{1}}\triangle\mathcal{D}_{S_{2}}=\{q^{2},p,p^{2}\}$. Note that $\phi(n/p)>\phi(n/p^{2})>\phi(n/q^{2})$. Recall (\ref{p2q21eq1}). Taking $\{q^{2},p,p^{2}\}$, $\phi(n/\cdot)$, $\{p,p^{2}\}$, $\mathcal{D}_{S_{1}}\setminus\mathcal{D}_{S_{2}}$, $\mathcal{D}_{S_{2}}\setminus\mathcal{D}_{S_{1}}$, as $A$, $f$, $B$, $A_{1}$, $A_{2}$, in Lemma \ref{notincontradiction}, we have
\begin{equation}\label{p2q21eq10}
\{p,p^{2}\}\nsubseteq\mathcal{D}_{S_{1}}\setminus\mathcal{D}_{S_{2}}\quad\text{and}\quad \{p,p^{2}\}\nsubseteq\mathcal{D}_{S_{2}}\setminus\mathcal{D}_{S_{1}}.
\end{equation}
Taking $\{q^{2},p,p^{2}\}$, $\phi(n/\cdot)$, $\{p,q^{2}\}$, $\mathcal{D}_{S_{1}}\setminus\mathcal{D}_{S_{2}}$, $\mathcal{D}_{S_{2}}\setminus\mathcal{D}_{S_{1}}$, as $A$, $f$, $B$, $A_{1}$, $A_{2}$, in Lemma \ref{notincontradiction}, we have
\begin{equation}\label{p2q21eq11}
\{p,q^{2}\}\nsubseteq\mathcal{D}_{S_{1}}\setminus\mathcal{D}_{S_{2}}\quad\text{and}\quad \{p,q^{2}\}\nsubseteq\mathcal{D}_{S_{2}}\setminus\mathcal{D}_{S_{1}}.
\end{equation}
Without loss of generality, we have
\begin{itemize}
\item Subcase 2.13: $\mathcal{D}_{S_{1}}\setminus\mathcal{D}_{S_{2}}=\{q^{2},p,p^{2}\}$ and $\mathcal{D}_{S_{2}}\setminus\mathcal{D}_{S_{1}}=\emptyset$\\
This contradicts (\ref{p2q21eq10}) and (\ref{p2q21eq11}).

\item Subcase 2.14: $\mathcal{D}_{S_{1}}\setminus\mathcal{D}_{S_{2}}=\{p,p^{2}\}$ and $\mathcal{D}_{S_{2}}\setminus\mathcal{D}_{S_{1}}=\{q^{2}\}$\\
This contradicts (\ref{p2q21eq10}).

\item Subcase 2.15: $\mathcal{D}_{S_{1}}\setminus\mathcal{D}_{S_{2}}=\{p,q^{2}\}$ and $\mathcal{D}_{S_{2}}\setminus\mathcal{D}_{S_{1}}=\{p^{2}\}$\\
This contradicts (\ref{p2q21eq11}).

\item Subcase 2.16: $\mathcal{D}_{S_{1}}\setminus\mathcal{D}_{S_{2}}=\{p^{2},q^{2}\}$ and $\mathcal{D}_{S_{2}}\setminus\mathcal{D}_{S_{1}}=\{p\}$\\
Then (\ref{p2q21eq1}) implies that $(p-2)q(q-1)=p(p-1)$. Since $p\geq3$ and $q(q-1)>p(p-1)$, we have $(p-2)q(q-1)>p(p-1)$, which is a contradiction.

\end{itemize}

\item Case 3: $|\mathcal{D}_{S_{1}}\triangle\mathcal{D}_{S_{2}}|=2$\\
Set $\mathcal{D}_{S_{1}}\triangle\mathcal{D}_{S_{2}}=\{a_{1},a_{2}\}$. Taking $\mathcal{D}_{S_{1}}\triangle\mathcal{D}_{S_{2}}$, $\phi(n/\cdot)$, $\mathcal{D}_{S_{1}}\setminus\mathcal{D}_{S_{2}}$, $\mathcal{D}_{S_{2}}\setminus\mathcal{D}_{S_{1}}$, as $A$, $f$, $A_{1}$, $A_{2}$, in Corollary \ref{notincontradictioncor}, we have either $\phi(n/a_{1})=\phi(n/a_{2})$ while $|\mathcal{D}_{S_{1}}\setminus\mathcal{D}_{S_{2}}|=|\mathcal{D}_{S_{2}}\setminus\mathcal{D}_{S_{1}}|=1$, or $\mathcal{D}_{S_{1}}\triangle\mathcal{D}_{S_{2}}=\emptyset$ leading to a contradiction. Note that $\phi(n/p)>\phi(n/p^{2})>\phi(n/q^{2})$ and that $\phi(n/p)>\phi(n/q)>\phi(n/q^{2})$. Without loss of generality, we have $\mathcal{D}_{S_{1}}\setminus\mathcal{D}_{S_{2}}=\{q\}$ and $\mathcal{D}_{S_{2}}\setminus\mathcal{D}_{S_{1}}=\{p^{2}\}$. Then (\ref{p2q21eq1}) implies that $q=p(p-1)$. Since $q$ is odd, we have a contradiction.

\item Case 4: $|\mathcal{D}_{S_{1}}\triangle\mathcal{D}_{S_{2}}|=1$\\
Recall (\ref{p2q21eq1}). Taking $\mathcal{D}_{S_{1}}\triangle\mathcal{D}_{S_{2}}$, $\phi(n/\cdot)$, $\mathcal{D}_{S_{1}}\setminus\mathcal{D}_{S_{2}}$, $\mathcal{D}_{S_{2}}\setminus\mathcal{D}_{S_{1}}$, as $A$, $f$, $A_{1}$, $A_{2}$, in Lemma \ref{notincontradictionnew}, we have $\mathcal{D}_{S_{1}}\triangle\mathcal{D}_{S_{2}}=\emptyset$, which is a contradiction.



\end{itemize}

This completes the proof.\qed

\end{proof}

\begin{theorem}\label{p2q2}
Set $n=p^{2}q^{2}$ with primes $3\leq p<q$. Let $\mathcal{D}_{S_{1}}$ and $\mathcal{D}_{S_{2}}$ be two subsets of $\mathcal{D}_{[n]}\setminus\{n\}$. Then $\spec(\ICG(n,\mathcal{D}_{S_{1}}))
=\spec(\ICG(n,\mathcal{D}_{S_{2}}))$ implies $\mathcal{D}_{S_{1}}=\mathcal{D}_{S_{2}}$.

\end{theorem}
\begin{proof}
Since $\phi(n)/n=\frac{(p-1)(q-1)}{pq}\geq\frac{2\cdot4}{3\cdot5}>\frac{1}{2}$, we have $\phi(n)>n/2$. By Lemma \ref{1notintriangle}, $1\notin\mathcal{D}_{S_{1}}\triangle\mathcal{D}_{S_{2}}$. By Lemma \ref{p2q21} and Corollay \ref{regularisocor}, we obtain the result.\qed

\end{proof}





\section{Conclusion}\label{conclusion}
In this work, we affirm So's conjecture for 4 types of integral circulant graphs. From our experience, it is difficult to completely solve So's conjecture and new methods should be involved. It is natural to discuss integral circulant graphs of order in other forms. However, without new techniques involved, it might be more complicated.

\section*{Acknowledgements}
The authors greatly appreciate the comments and suggestions from the anonymous referees.






\end{document}